\documentclass[psamsfonts]{amsart}
\usepackage[usenames,dvipsnames]{color}
\usepackage{comment,amsthm,amssymb,amsmath,mathrsfs,indentfirst,commutative-diagrams}
\usepackage[normalem]{ulem}
\usepackage{url}
\usepackage[all,arc,2cell]{xy}
\UseAllTwocells
\usepackage{enumerate}
\usepackage{hyperref}  
\hypersetup{%
  bookmarksnumbered=true,%
  bookmarks=true,%
  colorlinks=true,%
  linkcolor=blue,%
  citecolor=blue,%
  filecolor=blue,%
  menucolor=blue,%
  urlcolor=blue,%
  pdfnewwindow=true,%
  pdfstartview=FitBH}

\def\makeautorefname#1#2{\expandafter\def\csname#1autorefname\endcsname{#2}}
\def\equationautorefname~#1\null{(#1)\null}
\makeautorefname{footnote}{footnote}%
\makeautorefname{item}{item}%
\makeautorefname{figure}{Figure}%
\makeautorefname{table}{Table}%
\makeautorefname{part}{Part}%
\makeautorefname{appendix}{Appendix}%
\makeautorefname{chapter}{Chapter}%
\makeautorefname{section}{Section}%
\makeautorefname{subsection}{Section}%
\makeautorefname{subsubsection}{Section}%
\makeautorefname{theorem}{Theorem}%
\makeautorefname{thm}{Theorem}%
\makeautorefname{cor}{Corollary}%
\makeautorefname{lem}{Lemma}%
\makeautorefname{prop}{Proposition}%
\makeautorefname{pro}{Property}
\makeautorefname{conj}{Conjecture}%
\makeautorefname{defn}{Definition}%
\makeautorefname{notn}{Notation}
\makeautorefname{notns}{Notations}
\makeautorefname{rem}{Remark}%
\makeautorefname{quest}{Question}%
\makeautorefname{exmp}{Example}%
\makeautorefname{ax}{Axiom}%
\makeautorefname{claim}{Claim}%
\makeautorefname{ass}{Assumption}%
\makeautorefname{asss}{Assumptions}%
\makeautorefname{con}{Construction}%
\makeautorefname{prob}{Problem}%
\makeautorefname{warn}{Warning}%
\makeautorefname{obs}{Observation}%
\makeautorefname{conv}{Convention}%

\newtheorem{thm}{Theorem}[subsection]
\newtheorem{cor}{Corollary}[subsection]
\newtheorem{prop}{Proposition}[subsection]
\newtheorem{lem}{Lemma}[subsection]

\theoremstyle{definition}
\newtheorem{defn}{Definition}[subsection]

\newtheorem{exmp}{Example}[subsection]
\newtheorem{exmps}{Examples}[subsection]
\newtheorem{notn}{Notation}[subsection]
\newtheorem{notns}{Notations}[subsection]

\newtheorem{rem}{Remark}[subsection]
\newtheorem{warn}{Warning}[subsection]

\makeatletter
\let\c@obs=\c@thm
\let\c@cor=\c@thm
\let\c@prop=\c@thm
\let\c@lem=\c@thm
\let\c@prob=\c@thm
\let\c@con=\c@thm
\let\c@conj=\c@thm
\let\c@defn=\c@thm
\let\c@notn=\c@thm
\let\c@notns=\c@thm
\let\c@exmp=\c@thm
\let\c@ax=\c@thm
\let\c@pro=\c@thm
\let\c@ass=\c@thm
\let\c@warn=\c@thm
\let\c@rem=\c@thm
\let\c@sch=\c@thm
\let\c@equation\c@thm
\numberwithin{equation}{section}
\makeatother

\title{Ring operads and symmetric bimonoidal categories}
\author{Kailin Pan}
\date{}

\begin{document}

    \begin{abstract}
        We generalize the classical operad pair theory 
        to a new model for $E_\infty$ ring spaces, 
        which we call ring operad theory, and establish a connection
        with the classical operad pair theory, allowing
        the classical multiplicative infinite 
        loop machine to be applied to algebras 
        over any $E_\infty$ ring operad.
        As an application, we show that 
        classifying 
        spaces of symmetric bimonoidal categories are directly homeomorphic
        to certain $E_\infty$ ring spaces in the ring operad sense.
        Consequently, this provides an alternative construction
        from symmetric bimonoidal categories to classical $E_\infty$
        ring spaces.
        We also present a comparison between this construction and 
        the classical approach.
    \end{abstract}

    \maketitle

    \tableofcontents

    \section{Introduction}

    In \cite{may1972geometry}, Peter May established the operad 
    theory to detect zero spaces of spectra. 
    Roughly speaking, any $E_\infty$ space 
    (algebra over some $E_\infty$ operad) admits a group completion
    that has an $n$-fold delooping for any natural number $n$,
    and thus homotopy equivalent to the zero space of some 
    connective spectrum.
    Therefore, any $E_\infty$ space determines 
    a connected spectrum and thus a generalized cohomology
    theory. As an example, the classifying
    space of some symmetric monoidal category admits an
    $E_\infty$ space structure as shown in \cite{OperadsforSymmetricMonoidalCategories}.
    In particular, beginning with the category of 
    projective modules over some commutative ring $R$ with direct sum,
    we get the additive algebraic $K$-theory spectrum $K(R)$.

    In practice, instead of additive generalized cohomology theory,
    we are more interested in generalized cohomology theory 
    with a good cup-product structure, which is represented by an 
    $E_\infty $ ring spectrum. This motivates a 
    multiplicative infinite loop space theory to detect 
    zero spaces of $E_\infty $ ring spectra. 
    Several different models for 
    multiplicative infinite loop space theory have been established.
    We list some of them.
    \begin{itemize}
        \item[(1)] The operad pair theory \cite{may1977ring,may1982multiplicative}.
        \item[(2)] The infinite categorical multiplicative infinite loop space machine \cite{Gepner_2015}.
        \item[(3)] The moperad and bioperad theory \cite{may_moperad}.
    \end{itemize}
    Each of these three theories defines a notion $E_\infty$ ring space
    and constructs for any $E_\infty$ ring space a ring completion
    that is homotopic to the zero space of some $E_\infty$ ring spectrum.
    However, the comparison between these three definitions of 
    $E_\infty$ ring space is still unknown.

    As mentioned above, the additive algebraic $K$-theory spectrum $K(R)$
    comes from the symmetric monoidal category of 
    projective modules over some commutative ring $R$ with direct sum.
    It turns out that $K(R)$ admits an $E_\infty$ ring structure 
    that is closely related to the symmetric bimonoidal structure 
    of the category of projective $R$-modules with direct sum and tensor 
    product. Hence, we expect the classifying space of any 
    symmetric bimonoidal category admits an $E_\infty$ ring space
    structure.

    This holds for the infinite categorical machine. In the 
    bioperad theory, there exists a bioperad $\mathscr{P}^{bi}$
    in the category of small categories whose algebras are precisely
    bipermutative categories, so classifying spaces of 
    bipermutative categories are $E_\infty$ ring spaces in the 
    bioperad sense. As shown in \cite[Section VI.3]{may1977ring},
    any symmetric bimonoidal category is functorially equivalent
    to some bipermutative category, so the bioperad theory also 
    provides a construction of algebraic $K$-theory ring spectra
    from symmetric bimonoidal categories.

    However, a similar construction does not work directly in the classical
    operad pair theory. There is no operad pair in the category of small categories
    whose algebra category is either symmetric bimonoidal categories
    or bipermutative categories. A construction
    from bipermutative categories to $E_\infty$ ring spaces in the 
    operad pair sense is shown in 
    \cite{may1982multiplicative,TheconstructionofE_infringspacesfrombipermutativecategories},
    in which Peter May proves that the classifying of a bipermutative
    category is homotopy equivalent to some $E_\infty$ ring spaces in the 
    operad pair sense.

    In this paper, we generalize the classical operad pair theory 
    to a new model for $E_\infty$ ring spaces, which we call ring operads.
    This new theory provides an alternative construction
    from bipermutative
    categories to classical $E_\infty$ ring spaces.
    
    Informally, we develop a new definition of 
    ``$E_\infty$ ring spaces" using this new ring operad theory.
    It turns out that this new notion of ``$E_\infty$ ring spaces"
    coincides with the classical one in the operad pair sense
    up to homotopy, so we can apply the classical
    delooping machine to $E_\infty$ ring spaces in the new sense.
    As an application, we show that 
    classifying 
    spaces of all tight symmetric bimonoidal categories with strict 
    zero and unit elements are directly homeomorphic
    to some $E_\infty$ ring space in the new sense.
    In particular, we get an alternative construction
    from bipermutative
    categories to classical $E_\infty$ ring spaces. 

    To compare this new construction with the classical one, 
    recall that in \cite{may1982multiplicative}
    the classical construction is based on a passage from 
    bipermutative categories to $\mathscr{F}\wr\mathscr{F}$-categories.
    Here $\mathscr{F}\wr\mathscr{F}$ is a specific small category
    defined in Definition \ref{defnoffwrf}.
    This passage is done by first construct a 
    lax functor from $\mathscr{F}\wr\mathscr{F}$ to the category
    of small categories and then strictify it to get a 
    genuine functor, and we will show that in this new 
    construction, the ring operad theory gives a new 
    interpretation of this lax functor. See Remark \ref{comptwoconst}
    for detail.

    \subsection{Statement of results}

    In analogy with classical operad theory, 
    starting with any symmetric monoidal category, 
    we define a ring operad $\mathscr{C}$ to be a collection of objects 
    $\{\mathscr{C}(f)\}$ indexed not by natural numbers 
    but by some collection of polynomials, together with 
    structure maps satisfying several commutative diagrams. See
    Definition \ref{defoftingoperad}. We also define a notion 
    of algebras over ring operads in Definition \ref{Calgebra}
    compared with algebras over classical operads.

    Similarly, in the category of spaces, we call a 
    ring operad to be $E_\infty$ if its all components are 
    contractible and the structure maps satisfy several freeness and 
    cofibration conditions. Less formally, an algebra 
    over such an $E_\infty$ ring operad is precisely a topological
    space with two binary operations $+,\times$ such that 
    all unit, associativity, commutativity, and distributivity laws 
    hold up to all higher homotopies.

    The following comparison theorem plays an essential role in 
    the ring operad theory.
    \begin{thm}[\ref{thmcompare1}]\label{thmcomp1.1.1}
        The categories of algebras over any two $E_\infty$ ring operads $\mathscr{C},\mathscr{C}'$
        have equivalent homotopy categories. Moreover, any $\mathscr{C}$ algebra
        is homotopy equivalent to some $\mathscr{C}'$ algebra.
    \end{thm}

    To compare the ring operad theory with classical operad pair theory,
    we prove the following theorem.
    \begin{thm}[\ref{propringoperadandoperadpairsamealg}, \ref{propoperadpairEinftyimpliesringoperadEinfty}]
        For any operad pair $(\mathscr{C},\mathscr{G})$,
        there is a ring operad $\mathscr{R}_{\mathscr{C},\mathscr{G}}$
        which is $E_\infty$ when $(\mathscr{C},\mathscr{G})$
        is $E_\infty$ such that their categories of algebras are isomorphic.
    \end{thm}

    As a corollary, combining the above two theorems and the classical
    multiplicative infinite loop machine, we get a 
    multiplicative infinite loop machine for all $E_\infty$ ring operad.

    \begin{thm}[\ref{thmgpcom}]
        For any $E_\infty$ ring operad $\mathscr{C}$, there exists 
        a functor $\mathbb{E}$ from the category of $\mathscr{C}$-algebras to 
        the category of $E_\infty$ ring spectra such that 
        $$X\to \Omega^\infty\mathbb{E}(X)$$
        is a ring completion.
    \end{thm}

    As applications, we construct two ring operads in the category
    of small categories whose algebras are precisely 
    tight symmetric bimonoidal categories
    with strict zero and unit elements and bipermutative categories.
    
    \begin{thm}[\ref{thmsymbi}, \ref{propsymbiEinfty}, \ref{thmsynbigpcom}, \ref{thmbiperm}]
        There exists two ring operads in the category
        of small categories $\mathscr{S},\mathscr{P}$
        such that $\mathscr{S}$-algebras are precisely 
        tight symmetric bimonoidal categories
        with strict zero and unit elements and 
        $\mathscr{P}$-algebras are precisely bipermutative categories.
        Moreover, after applying the classifying space functor $B=|N(-)|$,
        both $B\mathscr{S}$ and $B\mathscr{P}$ are $E_\infty$.
        Therefore, the classifying space of a
        tight symmetric bimonoidal category
        with strict zero and unit elements admits a ring completion 
        to the zero space of some $E_\infty$ ring spectrum.
    \end{thm}

    Here tight means the distributivity laws are given by 
    natural isomorphisms. Moreover, this new construction coincides with the classical one.

    \begin{thm}[\ref{comptwoconst}]
        The two constructions from bipermutative categories to classical $E_\infty$ ring spaces
        coincide up to homotopy.
    \end{thm}

    With this operadic description of tight symmetric bimonoidal categories
    with strict zero and unit elements and bipermutative categories,
    we reprove the strictification theorem in \cite[Section VI.3]{may1977ring} from 
    symmetric bimonoidal categories to bipermutative categories.

    \begin{thm}[\ref{thmstrictification}]
        There is a functor $\Phi$ from the category of
        tight symmetric bimonoidal categories
        (with strict unit and zero objects) to the category of bipermutative categories 
        and a natural equivalence $\eta: \Phi C \to C$ of symmetric bimonoidal
        categories. 
    \end{thm}

    \subsection{Organization of the paper}

    We now summarize the contents of the paper. 
    
    In Section 
    \ref{sectionfundation}, we build up the foundations of ring operad theory.

    The precise definition of ring operad and algebras over ring operads
    are given in Section \ref{Sectiondefn}. 
    Recall that an operad can be regarded as a functor on the 
    category $\Sigma$ of finite sets 
    and symmetric groups together with some structure maps.
    However, in the ring operad theory,
    the 
    category $\Sigma$
    is generalized to a rather complicated category 
    $\widehat{\mathcal{R}}$. 

    In section \ref{sectioncomparison}, we define the 
    notion of $E_\infty$ ring operad and state
    the most fundamental and useful theorem in the theory of ring operads:
    the Comparison
    Theorem \ref{thmcomp1.1.1}. We defer its proof to 
    Appendix \ref{appendix}.
    
    More precisely, we analyze the combinatorial properties of $\widehat{\mathcal{R}}$
    in Section \ref{Sectioncomb}. In particular, we construct
    a filtration of a sub-category $\widehat{\mathcal{R}}_{n.d.}$
    of $\widehat{\mathcal{R}}$ in Proposition \ref{propofnondegen}.
    Using this filtration, we construct a filtration of 
    the monad $\mathbb{C}X$ associated to an $E_\infty$ ring operad 
    $\mathscr{C}$ in Section \ref{Sectiondiffoperad}. This 
    plays an essential role in the proof of the Comparison Theorem
    \ref{thmcomp1.1.1}.

    In Section 
    \ref{sectionclassical}, we compare ring operad theory with 
    two classical theories: operad pair theory and category of ring operator
    theory.

    In Section 
    \ref{Sectionoperadpair}, we construct for any ($E_\infty$) operad pair 
    $(\mathscr{C},\mathscr{G})$,
    a (an $E_\infty$) ring operad $\mathscr{R}_{\mathscr{C},\mathscr{G}}$
    such that their categories of algebras are isomorphic.
    Therefore, applying the Comparison Theorem \ref{thmcomp1.1.1},
    we get the multiplicative infinite loop machine Theorem \ref{thmgpcom} for 
    any $E_\infty$ ring operad.

    In Section 
    \ref{Sectionringoperator}, we compare ring operad theory with 
    category of ring operator
    theory that was used in the construction
    from bipermutative categories to $E_\infty$ ring spaces in the 
    operad pair sense in \cite{TheconstructionofE_infringspacesfrombipermutativecategories}.
    More precisely, we construct for any $E_\infty$ ring operad $\mathscr{C}$
    a category of ring operators $\tilde{\mathscr{C}}$
    such that their categories of algebras have the same homotopy categories.
    We summarize this by the following diagram in which all functors induce
    equivalences on homotopy categories.
    \begin{center}
        \begin{codi}
            \obj {|(1)| \mathscr{C}\text{-spaces}
            &[5em]&[5em] |(2)| \text{special }\tilde{\mathscr{C}}\text{-spaces}\\[-1em]
            |(3)| (\mathscr{C}\times\mathscr{R}_{\mathscr{K},\mathscr{L}}) \text{-spaces}
            &&|(4)| \text{special }\widetilde{(\mathscr{C}\times\mathscr{R}_{\mathscr{K},\mathscr{L}})}=(\tilde{\mathscr{C}}\times_{\mathscr{F}\wr\mathscr{F}}\tilde{\mathscr{R}}_{\mathscr{K},\mathscr{L}})\text{-spaces}\\[-1em]
            |(5)|  \mathscr{R}_{\mathscr{K},\mathscr{L}}\text{-spaces}
            &&|(6)| \text{special } \tilde{\mathscr{R}}_{\mathscr{K},\mathscr{L}}
            =(\hat{\mathscr{L}}\wr \hat{\mathscr{K}})\text{-spaces}\\[-1em]
            &|(7)| (\mathscr{K},\mathscr{L})\text{-spaces}&\\
            };
            \mor 1 "\nu":-> 2;
            \mor 3 "\nu":-> 4;
            \mor 5 "\nu":-> 6;
            \mor 1 -> 3;
            \mor 2 -> 4;
            \mor 5 -> 3;
            \mor 6 -> 4;
            \mor 7 -> 5;
            \mor 7 -> 6;
        \end{codi}
    \end{center}

    In Section \ref{sectioncategory}, we apply the ring operad theory
    to classical category theory.
    
    In Section \ref{Sectionsymbimonoidal}, we construct a ring operad
    $\mathscr{S}$ in the category
    of small categories whose algebras are precisely 
    tight symmetric bimonoidal categories (defined in 
    \cite[Volume I, Definition 2.1.2]{johnson2021bimonoidal})
    with strict zero and unit elements such that the classifying space  
    $B\mathscr{S}$ is $E_\infty$.
    Here tight means the distributivity maps are isomorphisms.
    Therefore, we get a multiplicative infinite loop machine 
    for tight symmetric bimonoidal categories
    with strict zero and unit elements.

    In Section \ref{Sectionbipermutative}, we mimic the above construction
    to get a ring operad $\mathscr{P}$ whose algebras are 
    bipermutative categories. 
    Hence, we get a multiplicative infinite loop machine 
    for bipermutative categories. 
    We state a comparison between this new construction and the classical
    one in Remark \ref{comptwoconst}.
    Lastly, we use this description to 
    reprove the strictification theorem from
    symmetric bimonoidal categories to bipermutative categories.

    \section{Foundations of ring operads}\label{sectionfundation}

    \subsection{Definition of ring operads}\label{Sectiondefn}
    Recall that in classical operad theory, to describe the higher homotopy,
    we put all operators that we want to identify up to higher homotopy
    into a contractible space. This strategy also works in the multiplicative
    context. Let $X$ be a space with two binary operators $+,\times$.
    In order to describe a higher homotopical distributivity law 
    $$a\times b+a\times c\simeq a\times(b+c),$$
    we construct a contractible space
    of trinary operators over $X$ which contains the above
    two trinary operators
    \begin{align*}
        (a,b,c)\to a\times b+a\times c,\\
        (a,b,c)\to a\times (b+c).
    \end{align*}
    Therefore, to describe all higher homotopical associativity, commutativity and 
    distributivity laws, we need a collection of contractible
    spaces of multivariable operators over $X$.

    Different from classical operad theory, this collection of 
    contractible spaces is not indexed by natural numbers but by
    polynomials.

    \begin{warn}
    Less formally, each polynomial $f$ determines a type 
    of multivariable operators over $X$ and then we can assign
    to each $f$ a space that parameterizes all 
    multivariable operator operators of that type. However,
    for some polynomials, the associated space is not 
    necessarily contractible. For example,
    consider a path in the space associated to $x+y$
    connecting
    $$(x,y)\mapsto x+y \text{ \ \ and \ \ } (x,y)\mapsto y+x.$$
    Then after evaluating both $x$ and $y$ to be $x$,
    the above path gives a loop in 
    the space associated to $2x$
    starting at $$x\mapsto x+x.$$
    This loop is not necessarily contractible in general.
    \end{warn}

    Therefore, to define a ring operad as a collection
    of contractible spaces indexed by polynomials,
    we expect the index set to only consist of
    those good polynomials whose associated spaces 
    are indeed contractible. It turns out the
    following collection of good polynomials is the largest 
    one that is closed under composition.

    \begin{notn}\label{rn}
        Fix a set of variables $\{a_{i,j}:1\leqslant i\leqslant j\}$
    and let $\mathbb{Z}[a_{1,n},a_{2,n},\cdots,a_{n,n}]$ for $n\geqslant0$
    be the polynomial ring on $n$ variables.
    Consider the subset 
    $\mathcal{R}(n)\subset\mathbb{Z}[a_{1,n},a_{2,n},\cdots,a_{n,n}]$, which
    consists of all polynomials $f$ such that $f$ is a finite sum of different
    monic monomials with positive degree and 
    each monomial in this sum is a product of different
    variables; that is, 
    $$\mathcal{R}(n)=\bigg\{\sum_{I=(i_1,i_2,\cdots,i_n)\in \{0,1\}^n\setminus \{0\}^n} 
    \varepsilon_I a_{1,n}^{i_1}\cdots a_{n,n}^{i_n}\in
    \mathbb{Z}[a_{n,1},\cdots,a_{n,n}] : \varepsilon_I = 0\text{ or } 1\bigg\}.$$
    To avoid confusion, let $0_n\in\mathcal{R}(n)$ denote the 
    zero element in $\mathcal{R}(n)$.
    We write $|f|=n$ for $f\in\mathcal{R}(n)$.
    \end{notn}
    
    Note that the collection of all monomials with positive degree that is a product of different
    variables is one-to-one corresponding to the set of non-empty subset of $n$ variables 
    $\{a_{1,n},a_{2,n},\cdots,a_{n,n}\}$, so $\mathcal{R}(n)$ is one-to-one
    corresponding to the power set $P(P\{a_{1,n},a_{2,n},\cdots,a_{n,n}\} \setminus \{\emptyset\})$,
    and thus $\mathcal{R}(n)$ is finite with cardinality $2^{2^n-1}$.

    There is also a non-symmetric operad structure on $\mathcal{R}(n)$ defined by 
    the composition of polynomials, which generalize the sum of natural
    numbers in the classical theory.
    \begin{notn}
        The composition of polynomials is given by
        \begin{gather*}
            \circ: \mathcal{R}(k)\times\mathcal{R}(j_1)\times\cdots\times\mathcal{R}(j_k)
            \to\mathcal{R}(j_+)\\
            (g,f_1,\cdots,f_k)\mapsto g(f_1(a_{j_+,1},\cdots,a_{j_+,j_1}),\cdots,f_k(a_{j_+,j_1+\cdots+j_{k-1}+1},\cdots,a_{j_+,j_+}))
        \end{gather*}
        where $j_+:=j_1+j_2+\cdots+j_k$. 
    \end{notn}

    \begin{lem}
        $\coprod \mathcal{R}(n)$ is closed under composition. 
    \end{lem}
    
    \begin{proof}
        By definition, a polynomial 
        $f\in\mathbb{Z}[a_{1,n},a_{2,n},\cdots,a_{n,n}]$ is contained in 
        $\mathcal{R}(n)$ if and only if the following conditions hold:
        \begin{itemize}
            \item[(1)] $\frac{\partial^2}{\partial a_{k,n}^2}f=0$ for each $k=1,2,\cdots,n$,
            \item[(2)] $$\frac{\partial^{|I|}}{\partial a_{I,n}}f(0,\cdots,0)=0\text{ or }1$$ 
                       for each 
                       sequence of finite length $I=(i_1,\cdots,i_l)$, 
                       $1\leqslant i_1<\cdots<i_l\leqslant n$, 
                       $l\geqslant1$,
            \item[(3)] $f(0,0,\cdots,0)=0$.
        \end{itemize}
        Using this criterion, we can check that $g(f_1,\cdots,f_k)$ is contained
        in $\mathcal{R}(j_+)$ and the associativity diagram commutes.
        
        Indeed, when $l=j_1+\cdots +j_{s-1}+r$, $ 1\leqslant r\leqslant j_{s}$,
        \begin{align*}
            &\frac{\partial^2}{\partial a_{l,j_+}^2}g(f_1,\cdots,f_k)\\
            =&\frac{\partial}{\partial a_{l,j_+}}\bigg(\frac{\partial}{\partial a_{s,k}}g(f_1,\cdots,f_k)\cdot \frac{\partial}{\partial a_{r,j_s}}f\bigg)\\
            =&\frac{\partial^2}{\partial a_{s,k}^2}g(f_1,\cdots,f_k)\cdot \bigg(\frac{\partial}{\partial a_{r,j_s}}f\bigg)^2+
            \frac{\partial}{\partial a_{s,k}}g(f_1,\cdots,f_k)\cdot \frac{\partial^2}{\partial a_{r,j_s}^2}f\\
            =&0.
        \end{align*}
    
        Also, when $I=(i_{1,1},\cdots,i_{r_1,1},\cdots, i_{1,k},\cdots,i_{r_k,k})$
        with $$j_1+\cdots +j_{s-1}< i_{1,s}< \cdots<
        i_{r_s,s}\leqslant j_1+\cdots +j_{s}$$
        for any $1\leqslant s \leqslant k$,
        we write $i'_{t,s}:=i_{t,s}-(j_1+\cdots +j_{s-1})$
        and $$L:=(l_1,\cdots,l_m)=\{s:j_s\geqslant1\}.$$ Then
        \begin{align*}
            &\frac{\partial^{|I|}}{\partial a_{I,j_+}}g(f_1,\cdots,f_k)(0,\cdots,0)\\
            =&\biggl(\frac{\partial^m}{\partial a_{L,k}}g(f_1,\cdots,f_k)\cdot
            \prod_{j_s\geqslant1} \frac{\partial^{r_{s}}}{\partial a_{i'_{1,s},j_s}\cdots\partial a_{i'_{r_{s},s},j_{s}}}f_{s}\biggr)(0,\cdots,0)\\
            =&0 \text{ or } 1.
        \end{align*}
        Here we need to use the fact that $f_s(0,\cdots,0)=0$.
    
        Lastly, we automatically have
        \begin{align*}
            g(f_1,\cdots,f_k)(0,\cdots,0)=0.
        \end{align*}
        Thus $\gamma$ gives a non-symmetric operad structure.
    \end{proof}

    The $\Sigma_j$ actions in the classical operad theory seem to be more
    difficult in the multiplicative context because there are more relations
    between these polynomials. To summarize these relations, we define 
    the following small category.
    
    \begin{notns}
        Let $\widehat{\mathcal{R}}$ be a small category with object set 
        $\displaystyle\coprod_{n\geqslant0}\mathcal{R}(n)$. 
        Let $\mathbf{Set}_e$ be the category of 
        set with two based points $\{0,e\}$ (that is the 
        under category of $\{0,e\}$) and 
        $$\mathbf{n}_e:=\{0,e,1,2\cdots,n\}$$ be 
        and object in $\mathbf{Set}_e$ for $n\geqslant0$.
    Note that for each map $\phi\in \mathbf{Set}_e(\mathbf{m}_e, \mathbf{n}_e)$
    the induced map 
    \begin{align*}
        \phi_*:\mathbb{Z}[a_{1,m},a_{2,m},\cdots,a_{m,m}]
        &\to\mathbb{Z}[a_{1,n},a_{2,n},\cdots,a_{n,n}]\\
        f(a_{1,m},a_{2,m},\cdots,a_{m,m})&\mapsto
        f(a_{\phi(1),n},a_{\phi(2),n},\cdots,a_{\phi(m),n})
    \end{align*}
    is functorial; that is $\phi_*\circ\psi_*=(\phi\circ\psi)_*$.
    Here $a_{0,n}=0$ and $a_{e,n}=1$.

    For each $f_m\in \mathcal{R}(m)$, $f_n\in \mathcal{R}(n)$, we define the hom-set to be
    \begin{align*}
        \widehat{\mathcal{R}}(f_m,f_n):=
        \{(f_m,\phi,f_n): \phi\in \mathbf{Set}_e(\mathbf{m}_e, \mathbf{n}_e),
        \phi_*(f_m)=f_n\}.
    \end{align*}
    The unit and composition are induced by those in $\mathbf{Set}_e$.

    We say a map $\phi\in \mathbf{Set}_e(\mathbf{m}_e, \mathbf{n}_e)$
    is effective if $\phi^{-1}(0)=\{0\},\phi^{-1}(e)=\{e\}$
    and denote the sub-category of effective morphisms
    by $ \widehat{\mathcal{R}}_{eff}$.
    Moreover, we say a map $\phi\in \mathbf{Set}_e(\mathbf{m}_e, \mathbf{n}_e)$
    singular if it is surjective
    and if $1\leqslant i < j \leqslant m$ with $\phi(i),\phi(j)\geqslant 1$,
    then $\phi(i)<\phi(j)$. 
    Then any $\phi\in \mathbf{Set}_e(\mathbf{m}_e, \mathbf{n}_e)$
    has a unique decomposition
    $\phi=p\circ\sigma$ such that $\sigma$ is singular
    and $p$ is effective. We say $\phi=p\circ\sigma$
    is the canonical decomposition of $\phi$.
    \end{notns}
    
    \begin{lem}\label{lemofinj}
        (1) For $f\in\mathbb{Z}[a_{1,m},a_{2,m},\cdots,a_{m,m}]$, if 
        $\phi_*f\in\mathcal{R}(n)$ for some injection $\phi:\mathbf{m}_e\to\mathbf{n}_e$,
        then $f\in\mathcal{R}(m)$.

        (2) For $(f_m,\phi,f_n)\in \widehat{\mathcal{R}}(f_m,f_n)$
        with $\phi=p\circ\sigma$ the canonical decomposition of $\phi$,
        $\sigma_*f_m\in\coprod \mathcal{R}(n)$.
    \end{lem}

    \begin{proof}
        When $\phi$ is injective, $\phi_*$ is also injective on monomials.
        Therefore, to show (1), it suffices to check this lemma for monomials $f$, which 
        holds immediately by definition.

        To see (2), we write $f_m$ as a sum of different monic 
        monomials $m_k$. Then it suffices to show that 
        nonzero elements in $\sigma_* m_k$ are 
        distinct monic monomials in $\coprod \mathcal{R}(n)$.

        Indeed, since $\sigma$ is singular, $\sigma_* m_k$
        is either $1$ or a 
        monic monomial in $\coprod \mathcal{R}(n)$ (probably zero).
        If $\sigma_* m_k=1$, then 
        $$p_*\sigma_*f_m(0,\cdots,0)=f_n(0,\cdots,0)\neq 0.$$
        If $\sigma_* m_k=\sigma_* m_l\not\in\{0,1\}$,
        then the coefficient of monomial $p_*\sigma_* m_k$
        in the standard decomposition of $f_n$ is at least $2$.
        Both contradict with $f_n\in\coprod \mathcal{R}(n)$.
    \end{proof}

    We will further focus on the properties of $\widehat{\mathcal{R}}_{eff}$ in 
    section \ref{Sectioncomb}.

    In the rest of this section, we modify the definitions in 
    classical operad theory to multiplicative context.
    Compare with \cite[Chapter 1,2]{may1972geometry}.

    Let $(\mathscr{V},\otimes,*)$ be a symmetric monoidal category such that 
    the unit object $*$ is also terminal.
    \begin{defn}\label{defoftingoperad}
        A ring operad $\mathscr{C}$ in $\mathscr{V}$ consists of a functor
        $\mathscr{C}:\widehat{\mathcal{R}}\to \mathscr{V}$,
        a unit map $\eta:*\to\mathscr{C}(a_{1,1})$,
        together with compositions
        $$\gamma:\mathscr{C}(g)\otimes\mathscr{C}(f_1)\otimes\cdots\otimes\mathscr{C}(f_k)\to
        \mathscr{C}(g(f_1,f_2,\cdots,f_k))$$ for all $f_1,f_2,\cdots,f_k\in
        \mathrm{Obj}(\widehat{\mathcal{R}})$, 
        $g\in\mathcal{R}(k)\subset\mathrm{Obj}(\widehat{\mathcal{R}})$,
        $k=1,2,\cdots$,
        such that $\mathscr{C}(0_n)\cong*$ for all $n\geqslant 0$ and
        the following diagrams commute:

        (1) Associativity diagram:

        For $g\in\mathcal{R}(k)$, $f_s\in\mathcal{R}(j_s)$, $s=1,2,\cdots,k$,
        $h_t\in\mathrm{Obj}(\widehat{\mathcal{R}})$, $t=1,2,\cdots,\Sigma j_s$,
        let $J_s:=j_1+\cdots+j_s$, $F_s:=f_s(h_{J_{s-1}+1},\cdots,h_{J_{s}})$,
        $G:=g(f_1,f_2,\cdots,f_k)$. Then we have 
        $G(h_1,\cdots,h_J)=g(F_1,\cdots,F_k)=:\tilde{G}$,
            \begin{center}
                \begin{codi}
                    \obj{
                        |(0)|\mathscr{C}(g)\otimes(\displaystyle\bigotimes_{s=1}^k\mathscr{C}(f_s))\otimes(\displaystyle\bigotimes_{t=1}^j\mathscr{C}(h_t))
                        &[12em] |(1)| \mathscr{C}(G)\otimes(\displaystyle\bigotimes_{t=1}^j\mathscr{C}(h_t))\\[2em]
                        |(0')|\mathscr{C}(g)\otimes\displaystyle\bigotimes_{s=1}^k(\mathscr{C}(f_s)\otimes(\displaystyle\bigotimes_{t=1}^{j_s}\mathscr{C}(h_{J_{s-1}+t})))
                        & \\[2em]
                        |(2)|\mathscr{C}(g)\otimes\displaystyle\bigotimes_{s=1}^k(\mathscr{C}(F_s))
                        &|(3)| \mathscr{C}(\tilde{G})\\
                        };
                    \mor 0 \gamma\otimes\mathrm{id}:-> 1;
                    \mor 0 "\text{shuffle}":-> 0';
                    \mor 0' "\mathrm{id}\otimes\gamma^k":-> 2;
                    \mor 1 \gamma:-> 3;
                    \mor 2 \gamma:-> 3;
                \end{codi}
            \end{center}

        (2) Unit diagrams:
            
        For $g\in\mathcal{R}(k) $,
            \begin{center}
                \begin{codi}
                    \obj{
                        |(0)|\mathscr{C}(g)\otimes(*)^{\otimes k}&[3em] |(1)|\mathscr{C}(g)\\
                        |(2)|\mathscr{C}(g)\otimes(\mathscr{C}(a_{1,1}))^{\otimes k}&\\
                        };
                    \mor 0 \cong:-> 1;
                    \mor[left] 0 "\text{id}\otimes\eta^k":-> 2;
                    \mor[below right] 2 "\gamma":-> 1;
                \end{codi}
            \end{center}
        and
            \begin{center}
                \begin{codi}
                    \obj{
                        |(0)|*\otimes\mathscr{C}(g)&[3em] |(1)|\mathscr{C}(g)\\
                        |(2)|\mathscr{C}(a_{1,1})\otimes\mathscr{C}(g)&\\
                        };
                    \mor 0 \cong:-> 1;
                    \mor[left] 0 "\eta\otimes\text{id}":-> 2;
                    \mor[below right] 2 "\gamma":-> 1;
                \end{codi}
            \end{center}
        (3) Equivariance diagrams:

        Let $g_m\in\mathcal{R}(m)$, $g_n\in\mathcal{R}(n)$, 
        $f_1\in\mathcal{R}(j_1)$, $h_1\in\mathcal{R}(r_1)$, $\cdots$, $f_n\in\mathcal{R}(j_n)$, $h_n\in\mathcal{R}(r_n)$
        and $(g_m,\psi,g_n)\in\mathrm{Mor}(\widehat{\mathcal{R}})$.
        When $\phi^{-1}(e)=e$, let $f_0=0_0$. Then,
            \begin{center}
                \begin{codi}
                    \obj {|(0)|\mathscr{C}(g_m)\otimes\mathscr{C}(f_1)\otimes\cdots\otimes\mathscr{C}(f_n)
                    &[12em]|(1')| \mathscr{C}(g_n)\otimes\mathscr{C}(f_1)\otimes\cdots\otimes\mathscr{C}(f_n) \\
                    |(1)| \mathscr{C}(g_m)\otimes\mathscr{C}(f_1)^{\otimes|\psi^{-1}(1)|}\otimes\cdots\otimes\mathscr{C}(f_n)^{\otimes|\psi^{-1}(n)|}&\\
                    |(2)| \mathscr{C}(g_m)\otimes\mathscr{C}(f_{\psi(1)})\otimes\cdots\otimes\mathscr{C}(f_{\psi(m)})&\\
                    |(3)| \mathscr{C}(g_m(f_{\psi(1)},\cdots,f_{\psi(m)}))
                    &|(3')| \mathscr{C}(g_n(f_1,f_2,\cdots,f_n)) \\};
                    \mor 0 "\mathscr{C}(\psi)\otimes \text{id}":-> 1';
                    \mor 0 "\text{id}\otimes\Delta":-> 1;
                    \mor 1 "\text{shuffle}":-> 2;
                    \mor 2 "\gamma":-> 3;
                    \mor 1' "\gamma":-> 3';
                    \mor 3 "\text{id}":-> 3';
                \end{codi}
            \end{center}

        When $\phi$ is singular with $\phi^{-1}(0)=0$, 
        let $f_e=a_{1,1}$. Then 
            \begin{center}
                \begin{codi}
                    \obj {|(0)|\mathscr{C}(g_m)\otimes\mathscr{C}(f_1)\otimes\cdots\otimes\mathscr{C}(f_n)
                    &[12em]|(1')| \mathscr{C}(g_n)\otimes\mathscr{C}(f_1)\otimes\cdots\otimes\mathscr{C}(f_n) \\
                    |(2)| \mathscr{C}(g_m)\otimes\mathscr{C}(f_{\psi(1)})\otimes\cdots\otimes\mathscr{C}(f_{\psi(m)})&\\
                    |(3)| \mathscr{C}(g_m(f_{\psi(1)},\cdots,f_{\psi(m)}))
                    &|(3')| \mathscr{C}(g_n(f_1,f_2,\cdots,f_n)) \\};
                    \mor 0 "\mathscr{C}(\psi)\otimes \text{id}":-> 1';
                    \mor 0 "\text{id}\otimes\eta":-> 2;
                    \mor 2 "\gamma":-> 3;
                    \mor 1' "\gamma":-> 3';
                    \mor 3 "\mathscr{C}(\tilde{\psi})":-> 3';
                \end{codi}
            \end{center}
        Here we let $\phi^{-1}(e)=\{e\}\cup\{i_1<\cdots<i_s\}$
        and let $j'_t:=j_{\psi(t)}$ for $t=1,2,\cdots,m$ 
        with $j_e=1$. Then we define 
        $\tilde{\psi}:\mathbf{j'_1+\cdots+j'_m}\to \mathbf{j_1+\cdots+j_n}$
        by sending 
        $$\sum_{l=1}^{i_k}j'_{l}\mapsto e$$
        for $k=1,\cdots,s$
        and sending 
        other elements bijectively order-preservingly
        to $\{1,\cdots,n\}\subset\mathbf{n}_e$.

        For $(f_1,\phi_1,h_1),\cdots,(f_n,\phi_n,h_n)\in\mathrm{Mor}(\widehat{\mathcal{R}})$,
            \begin{center}
                \begin{codi}
                    \obj {|(1)|\mathscr{C}(g_n)\otimes\mathscr{C}(f_1)\otimes\cdots\otimes\mathscr{C}(f_k)&[10em]
                          |(2)| \mathscr{C}(g_n(f_1,\cdots,f_k)) \\
                          |(1')|\mathscr{C}(g_n)\otimes\mathscr{C}(h_1)\otimes\cdots\otimes\mathscr{C}(h_k)
                          & |(2')| \mathscr{C}(g_n(h_1,\cdots,h_k))\\};
                    \mor 1 "\text{id}\otimes\mathscr{C}(\phi_1)\otimes\cdots\otimes\mathscr{C}(\phi_k)":-> 1';
                    \mor 2 "\mathscr{C}(\phi_1\oplus\cdots\oplus \phi_k)":-> 2';
                    \mor 1 "\gamma":-> 2;
                    \mor 1' "\gamma":-> 2';
                \end{codi}    
            \end{center}
    
        Here the diagonal morphisms $\mathscr{C}(f_s)\to \mathscr{C}(f_s)^{\otimes |\psi^{-1}(s)|}$ is defined to be 
        the counit $\mathscr{C}(f_s)\to *$ (recall that $*$ is terminal) if $\psi^{-1}(s)$ is empty, 
        and 
        \begin{align*}
            \phi_1\oplus\cdots\oplus \phi_k:&\mathbf{j_1+\cdots+j_n}\to \mathbf{r_1+\cdots+r_n}\\
            x+\sum_{t=1}^kj_{t}&\mapsto\begin{cases}
                \phi_{k+1}(x)+\sum_{t=1}^k r_t &\text{ if } 1\leqslant x\leqslant j_{k+1} \text{ and } \phi_{k+1}(x)\not\in\{0,e\}, \\
                \phi_{k+1}(x) &\text{ if } \phi_{k+1}(x)\in\{0,e\} \\
            \end{cases}
        \end{align*}
        defines morphism
        \begin{align*}
            (g_n(f_1,\cdots,f_k),\phi_1\oplus\cdots\oplus \phi_n,g_n(h_1,\cdots,h_k))&\in\mathrm{Mor}(\widehat{\mathcal{R}}).
        \end{align*}

        Moreover, a morphism between two ring operads is a natural transformation
        between functors which preserves all other structure maps.
    \end{defn}
        
    From now on, we assume our ground category 
    $\mathscr{V}$ to be cocomplete and also assume that 
    the symmetric monoidal structure $\bullet\otimes\bullet$
    is cocontinuous on each variable. 
    Let $\kappa_0, \kappa_e$ be two copies of $*$
    and let $S^0:=\kappa_0\coprod\kappa_e$.
    We define $\mathscr{V}_e$ to be the under category of $\mathscr{V}$
    with respect to $S^0$; 
    that is the category of objects $X$ in $\mathscr{V}$
    with two specific morphisms 
    \begin{align*}
        \epsilon: \kappa_e &\to X\\
        \iota: \kappa_0 &\to X
    \end{align*}
    and morphisms that preserve $\epsilon$ and $\iota$.

    \begin{rem}
        In the ring operad context, slightly different from the classical operad theory, 
        we regard a ring operad $\mathscr{C}$ as a covariant functor from 
        $\widehat{\mathcal{R}}$ and $X^{\otimes n}$ for some 
        $(X,\epsilon,\iota)$ in $\mathscr{V}$ a contravariant functor
        by shuffling the components; that is, for 
        $(f_m,\phi,f_n)\in\widehat{\mathcal{R}}(f_m,f_n)$, $f_m\in\mathcal{R}(m)$,
        $f_n\in\mathcal{R}(n)$,
        \begin{align*}
            \phi^*: X^{\otimes n}=X_1\otimes X_2 \otimes \cdots\otimes X_n \to X_{\phi(1)}\otimes X_{\phi(2)} \otimes \cdots\otimes X_{\phi(m)}\to X^{\otimes m},
        \end{align*}
        Here $X_0=\kappa_0$, $X_e=\kappa_e$ with 
        \begin{align*}
            \epsilon: X_e &\to X\\
            \iota: X_0 &\to X
        \end{align*}
        the structure maps of $X$.
    \end{rem}

    \begin{defn}\label{Calgebra}
        Let $\mathscr{C}$ be a ring operad. A $\mathscr{C}$-algebra is 
        an object $(X,\epsilon,\iota)$ in $\mathscr{V}_e$ together with morphisms
        in $\mathscr{V}$
        $$\theta:\mathscr{C}(f)\otimes X^{\otimes j}\to X$$
        for $f\in\mathcal{R}(j)$ 
        such that the following diagrams commute in $\mathscr{V}$:

        (1) Associativity diagram:

        For $g\in\mathcal{R}(k)$, $f_s\in\mathcal{R}(j_s)$, $s=1,2,\cdots,k$,
        let $J:=j_1+\cdots+j_k$,
            \begin{center}
                \begin{codi}
                    \obj{
                        |(0)|\mathscr{C}(g)\otimes(\displaystyle\bigotimes_{s=1}^k\mathscr{C}(f_s))\otimes X^{\otimes J}
                        &[12em] |(1)| \mathscr{C}(g(f_1,f_2,\cdots,f_k))\otimes X^{\otimes J}\\[2em]
                        |(0')|\mathscr{C}(g)\otimes\displaystyle\bigotimes_{s=1}^k(\mathscr{C}(f_s)\otimes X^{\otimes j_s})
                        & \\[2em]
                        |(2)|\mathscr{C}(g)\otimes X^{\otimes k}
                        &|(3)| X\\
                        };
                    \mor 0 \gamma\otimes\mathrm{id}:-> 1;
                    \mor 0 "\text{shuffle}":-> 0';
                    \mor 0' "\mathrm{id}\otimes\theta^k":-> 2;
                    \mor 1 \theta:-> 3;
                    \mor 2 \theta:-> 3;
                \end{codi}
            \end{center}

        (2) Unit diagram:
            
        For $g\in\mathcal{R}(k) $,
            \begin{center}
                \begin{codi}
                    \obj{
                        |(0)|*\otimes X&[3em] |(1)|X\\
                        |(2)|\mathscr{C}(a_{1,1})\otimes X&\\
                        };
                    \mor 0 \cong:-> 1;
                    \mor[left] 0 "\eta\otimes\text{id}":-> 2;
                    \mor[below right] 2 "\theta":-> 1;
                \end{codi}
            \end{center}
            
        (3) Equivariance diagram:

        For $g_m\in\mathcal{R}(m)$, $g_n\in\mathcal{R}(n)$, 
        $(g_m,\psi,g_n)\in\mathrm{Mor}(\widehat{\mathcal{R}})$,
            \begin{center}
                \begin{codi}
                    \obj {|(0)|\mathscr{C}(g_m)\otimes X^{\otimes n}
                    &[12em]|(1')| \mathscr{C}(g_n)\otimes X^{\otimes n} \\
                    |(2)| \mathscr{C}(g_m)\otimes X^{\otimes m}&|(3')| X \\
                    };
                    \mor 0 "\mathscr{C}(\psi)\otimes \text{id}":-> 1';
                    \mor 0 "\psi^*":-> 2;
                    \mor 2 "\theta":-> 3';
                    \mor 1' "\theta":-> 3';
                \end{codi}
            \end{center}
    \end{defn}

    The category of all $\mathscr{C}$-algebras
    will be denoted by $\mathscr{C}[\mathscr{V}_e]$.

    Then we can define a monad 
    $\mathbb{C}$ associated to any ring operad $\mathscr{C}$ in 
    $\mathscr{V}_e$.

    \begin{defn}\label{defnofmonad}
        Let $\mathscr{C}$ be a ring operad in $\mathscr{V}$.
        We define the monad $\mathbb{C}:\mathscr{V}_e\to \mathscr{V}_e$
        associated to $\mathscr{C}$ as the following coend. 
        For any $(X,\epsilon,\iota)\in\mathscr{V}_e$,
        \begin{align*}
            \mathbb{C}X:=&\mathscr{C}(\bullet)\otimes_{\widehat{\mathcal{R}}^{op}}X^\bullet.
        \end{align*}

        The unit and composition are induced by 
        \begin{align*}
            \eta:& X=*\otimes X\to \mathscr{C}(a_{1,1})\otimes X \to \mathbb{C}X,\\
            \mu:& \mathscr{C}(f)\otimes \bigotimes_i(\mathscr{C}(g_i)\otimes X^{|g_i|})\to \mathscr{C}(f(g_1,\cdots,g_{|f|}))\otimes X^{|g_1|+\cdots+|g_{|f|}|}.
        \end{align*}
        The structure map $S^0\to \mathbb{C}X$
        is the unique one such that $\eta: X\to \mathbb{C}X$
        is a morphism in $\mathscr{V}_e$.

        We denote the category of $\mathbb{C}$-algebras in $\mathscr{V}$
        by $\mathbb{C}[\mathscr{V}_e]$.
    \end{defn}

    \begin{lem}
        $\mathbb{C}:\mathscr{V}_e\to \mathscr{V}_e$ is a 
        well-defined monad.
    \end{lem}

    \begin{proof}
        This holds immediately from Definition \ref{defoftingoperad}.
        More precisely, the multiplication $\mu$ is well-defined
        by the equivariance diagrams, the associativity and unit diagrams
        \begin{center}
            \begin{codi}
                \obj {|(0)|\mathbb{C}\mathbb{C}\mathbb{C}X
                &|(1')| \mathbb{C}\mathbb{C}X
                & |(a)| \mathbb{C}X
                & |(b)| \mathbb{C}\mathbb{C}X
                & |(c)| \mathbb{C}X\\
                |(2)| \mathbb{C}\mathbb{C}X&|(3')| \mathbb{C}X
                & & |(d)| \mathbb{C}X & \\
                };
                \mor 0 "\mu":-> 1';
                \mor[left] 0 "\mathbb{C}\mu":-> 2;
                \mor 2 "\mu":-> 3';
                \mor 1' "\mu":-> 3';
                \mor a "\eta":-> b "\mu":-> d;
                \mor[above] c "\mathbb{C}\eta":-> b;
                \mor[below left] a "id":-> d;
                \mor c "id":-> d;
            \end{codi}
        \end{center}
        commute from the associativity and unit diagrams
        in Definition \ref{defoftingoperad}.
    \end{proof}
    
    \begin{prop}\label{propofmonadandoperad}
        Let $\mathscr{C}$ be a ring operad in $\mathscr{V}$
        with associated monad $\mathbb{C}$. Then 
        their categories of algebras are isomorphic; that is,
        $$\mathbb{C}[\mathscr{V}_e]\cong \mathscr{C}[\mathscr{V}_e].$$
    \end{prop}

    \begin{proof}
        The proof of the corresponding proposition
        in the classical operad theory also works here, see 
        \cite[Proposition 2.8]{may1972geometry}.

        In short, a morphism $\theta:\mathbb{C}X\to X$
        is precisely a collection of morphisms 
        $\theta_f: \mathscr{C}(f)\otimes X^{|f|}\to X$ 
        such that the equivariance diagram in Definition 
        \ref{Calgebra} commutes. Moreover, such a morphism $\theta$
        defines a $\mathbb{C}$-algebra structure on $X$ if and only if 
        the associativity and unit diagrams in Definition 
        \ref{Calgebra} commute.
    \end{proof}

    \begin{exmp}
        \emph{Strict ring operads}. Let $\mathscr{R}_{st}$ be a ring operad
         with $\mathscr{R}_{st}(g)=*$ for each 
        $g\in\mathrm{Obj}(\widehat{\mathcal{R}})$. The image of 
        morphisms in $\widehat{\mathcal{R}}$ is determined uniquely 
        since $*$ is terminal. This ring operad is called the strict ring operad
        based in $\mathscr{V}$ because algebras over $\mathscr{R}_{st}$ are precisely strict rig objects
        in $\mathscr{V}$. Also, note that each component $\mathscr{R}_{st}(g)$ 
        is terminal, so $\mathscr{R}_{st}$ itself is terminal in the category of 
        ring operads based in $\mathscr{V}$.
    \end{exmp}

    \subsection{The Comparison Theorem}\label{sectioncomparison}
    In this section, we define a notion of $E_\infty$ ring operad whose algebras,
    as desired, are spaces with addition and multiplication satisfying 
    higher homotopical associativity, commutativity and 
    distributivity laws. Then we state the Comparison
    Theorem \ref{thmcompare1} which says the categories of algebras over any two 
    $E_\infty$ ring operad have equivalent homotopy categories.
    The Comparison Theorem is the most fundamental and useful theorem
    in the theory of ring operads.
    However, the proof of this theorem is not easy to 
    understand at first reading due to the extensive 
    combinatorial details it contains, and since it 
    will not be used elsewhere, we will defer it to 
    Appendix \ref{appendix}.

    We first introduce some notations.

    \begin{defn}\label{special}
        (1) A polynomial $f\in\mathcal{R}(n)$ is called non-degenerate if 
        $\frac{\partial}{\partial a_{n,s}}f\neq 0$ 
        for each $s=1,2,\cdots,n$ and it is 
        called degenerate otherwise.
        
        (2) We say a non-degenerate
        $f\in\mathcal{R}(n)$ is special if it is of the form 
        $$f=a_{1,n}\cdots a_{k_1,n}+a_{k_1+1,n}\cdots a_{k_1+k_2,n}+\cdots+a_{k_1+\cdots+k_{l-1}+1,n}\cdots a_{k_1+\cdots+k_{l},n}$$
        for some $0<k_1\leqslant k_2\leqslant \cdots\leqslant k_l$
        with $k_1+k_2+\cdots+k_l=n$.
    \end{defn}

    Now, we fix our ground
    category to be the category of unbased topological spaces
    $(\mathscr{U},\times,*)$.

    \begin{defn}\label{defnEinfty}
        When the ground category $(\mathscr{V},\otimes,*)$ is 
        the category of unbased spaces
        $(\mathscr{U},\times,*)$,
        an $E_\infty$ ring operad $\mathscr{C}$
        is a ring operad such that:

        (1) All $\mathscr{C}(f)$ are contractible.
        
        (2) Let $(f_m,\phi,f_n)\in\widehat{\mathcal{R}}(f_m,f_n)$
        be a morphism such that $\phi:\mathbf{m}\to \mathbf{n}$
        is an injection as a map between sets. Then 
        $\phi_*:\mathscr{C}(f_m)\to \mathscr{C}(f_n)$ is a homeomorphism.

        (3) Given non-degenerate objects
        $f_1,f_2$ and $\alpha_1\in\mathscr{C}(f_1),\alpha_2\in\mathscr{C}(f_2)$,
        if there exists some non-degenerate $g$ with effective 
        morphisms
        $(f_1,\phi_1,g), (f_2,\phi_2,g)$ in $\widehat{\mathcal{R}}_{eff}$
        such that $$\phi_{1*}\alpha_1=\phi_{2*}\alpha_2\in\mathscr{C}(g),$$
        then there must exist a non-degenerate $h$ with effective morphisms
        $$(h,\psi_1,f_1), (h,\psi_2,f_2)$$ in $\widehat{\mathcal{R}}_{n.d.}$
        and a $\beta\in\mathscr{C}(h)$
        such that 
        \begin{align*}
            \psi_{1*}\beta&=\alpha_1,\\
            \psi_{2*}\beta&=\alpha_2.
        \end{align*}

        (4) Given any two effective morphisms 
        $$(f,\phi_1,g),(f,\phi_2,g)$$
        with $f$ non-degenerate,
        if there exists $\alpha\in\mathscr{C}(f)$ such that $\phi_{1*}\alpha=\phi_{2*}\alpha$,
        then $\phi_1=\phi_2$.
        
        (5) For any morphism $(f,\phi,g)$ in $\widehat{\mathcal{R}}_{n.d.}$,
        $\phi_*$ is a cofibration.

    \end{defn}

    In this definition, condition (1) arises from our motivation
    to require elements in $\mathscr{C}(f)$ represent higher homotopically
    equivalent multivariable operators. Condition (2) arises from
    the interpretation that only those spaces indexed by non-degenerate
    objects provide essential information, as we have shown in previous 
    sections. Condition (5) is a technical requirement which is necessary 
    in the proof of the Comparison Theorem. 

    Condition (3) and (4) are generalizations of the $\Sigma_j$ free condition 
    in the classical operad theory. Moreover, these conditions imply
    the following result.
    
    \begin{lem}
        Let $\mathscr{C}$ be an $E_\infty$ ring operad. Then 
        $\Sigma_n$ acts (on the left) freely on 
        $$\coprod_{f\in\mathcal{R}_{n.d.}(n)}\mathscr{C}(f).$$
    \end{lem}
    \begin{proof}
        By Lemma \ref{lemofnondeg}, for any $f\in\mathcal{R}_{n.d.}(n)$
        and $\phi\in\Sigma_n$, $\phi_*f\in \mathcal{R}_{n.d.}(n)$,
        so the $\Sigma_n$ action is well defined. The freeness follows
        directly from condition (4).
    \end{proof}

    Now we state the Comparison Theorem.

    \begin{thm}[Comparison Theorem]\label{thmcompare1}
        Let $\mathscr{C},\mathscr{C'}$ be any two 
        $E_\infty$ ring operads. Then the homotopy categories
        of
        $\mathscr{C}'[\mathscr{U}_e]$ and $\mathscr{C}[\mathscr{U}_e]$
        are equivalent.

        Moreover, any $\mathscr{C}$-algebra $X$ is equivalent
        to some $\mathscr{C}'$-algebra $Y$.
    \end{thm}

    \section{Comparison with classical theories}\label{sectionclassical}

    \subsection{Ring operads and operad pairs}\label{Sectionoperadpair}

    In this section, we construct a ring operad $\mathscr{R}_{\mathscr{C},\mathscr{G}}$ from any classical
    operad pair $(\mathscr{C},\mathscr{G})$ such that their 
    categories of algebras coincide. See 
    \cite{TheconstructionofE_infringspacesfrombipermutativecategories}
    for the definition of an operad pair.
    Then we show that the ring operad 
    $\mathscr{R}_{\mathscr{C},\mathscr{G}}$ is $E_\infty$ when 
    $(\mathscr{C},\mathscr{G})$ is an $E_\infty$ operad pair.
    Therefore, the Comparison Theorem \ref{thmcompare1} 
    generalizes the classical multiplicative infinite loop machine
    to be applied on algebras
    over any $E_\infty$ ring operad. 

    Before the definition, we need some notations first.

    \begin{notns}\label{notnofgamma}
        (1) For any $$I=(i_1,\cdots,i_n)\in\{0,1\}^n,$$ let $\Gamma_I$
        be the totally ordered set $\{j\in\{1,2,\cdots,n\}:i_j=1\}$
        with order induced by that of integers. Then 
        \begin{align*}
            \{0,1\}^n &\longleftrightarrow P(\{1,2,\cdots,n\})\\
            I &\longleftrightarrow \Gamma_I
        \end{align*}
        is a bijection.

        (2) For any $$f=\sum_{I=(i_1,i_2,\cdots,i_n)\in \{0,1\}^n} 
        \varepsilon_I a_{1,n}^{i_1}\cdots a_{n,n}^{i_n}\in \mathcal{R}(n),$$
        let $\Lambda_f$ be the totally ordered set 
        $\{I=(i_1,\cdots,i_n)\in\{0,1\}^n:\varepsilon_I=1\}$
        with lexicographical order. Then
        \begin{align*}
            \mathcal{R}(n) &\longleftrightarrow P(\{0,1\}^n)\\
            f &\longleftrightarrow \Lambda_f
        \end{align*}
        is a bijection.

    \end{notns}
    \begin{lem}\label{type}
        Let $(f_m,\phi,f_n)\in\widehat{\mathcal{R}}(f_m,f_n)$ be a morphism, say
            $$f_m=\sum_{I=(i_1,i_2,\cdots,i_m)\in \{0,1\}^m} 
        \varepsilon_I a_{1,m}^{i_1}\cdots a_{m,m}^{i_m}\in \mathcal{R}(m).$$
        If $\phi$ is effective, then there is a bijection
        $\tilde{\phi}: \Lambda_{f_m}\to \Lambda_{f_n}$ such that
        for each $I\in \Lambda_{f_m}$,
        \begin{align*}
            \phi|_{\Gamma_I}: \Gamma_I &\to \Gamma_{\tilde{\phi}(I)}\\
            j&\mapsto \phi(j)
        \end{align*}
        is a bijection.

        In general, if $\phi$ is not necessarily effective, 
        then 
        $\phi$ induces an injection 
        $\phi': \Lambda_{f_n}\to \Lambda_{f_m}$
        sending each monomial summand $m\in\Lambda_{f_n}$
        to the unique monomial summand $m'\in\Lambda_{f_m}$
        such that $\phi_*m'=m$.

        Moreover, for each $J\in \Lambda_{f_n}$, 
        $\phi$ induces an injection 
        $\phi_{J}: \Gamma_J\to \Gamma_{\phi'J}$
        defined by the restriction of $\phi^{-1}$.
    \end{lem}
    \begin{proof}
        By definition, 
        \begin{align*}
            f_n=&\sum_{I=(i_1,i_2,\cdots,i_m)\in \{0,1\}^m} 
        \varepsilon_I a_{\phi(1),n}^{i_1}\cdots a_{\phi(m),n}^{i_n}\\
        =&\sum_{J=(j_1,j_2,\cdots,j_n)\in \{0,1\}^n} 
        \varepsilon_J a_{1,n}^{j_1}\cdots a_{n,n}^{j_n}\in \mathcal{R}(n).
        \end{align*}

        The second line of the above expression is the decomposition
        of $f_n$ under the monomial basis, in which all coefficients are zero or one.
        The first line is also a sum of monomials, so the fact that these two expressions
        of $f_n$ coincide implies the above lemma.
    \end{proof}

    Here is a useful property of $\widehat{\mathcal{R}}_{eff}$.

    \begin{lem}\label{lemofspecialpullback}
        Let $f$ be a special object.
        Then for any two effective morphisms $(f,\phi_1,h)$, $(f,\phi_2,h)$,
        there exists an automorphism $(f,\sigma,f)$ such that 
        $\phi_1\sigma=\phi_2$.
    \end{lem}
    \begin{proof}
        Assume
        $$f=a_{1,n}\cdots a_{k_1,n}+a_{k_1+1,n}\cdots a_{k_1+k_2,n}+\cdots+a_{k_1+\cdots+k_{l-1}+1,n}\cdots a_{k_1+\cdots+k_{l},n}.$$
        
        Note that if $k_i=k_j$, then the following permutation
        $$(k_1+\cdots+k_{i-1}+1,k_1+\cdots+k_{j-1}+1)\circ\cdots\circ(k_1+\cdots+k_{i-1}+k_i,k_1+\cdots+k_{j-1}+k_j)$$
        induces an automorphism of $f$.

        Therefore, without lost of generality, we assume the induced maps $\phi_{1*}=\phi_{2*}:\Lambda_f\to\Lambda_h$
        coincide. 

        Also, note that any permutation of $\{k_1+\cdots+k_{i-1}+1,\cdots,k_1+\cdots+k_{i-1}+k_i\}$
        induces an automorphism of $f$ for all $i$, so there exists an 
        automorphism $(f,\sigma,f)$ induced by a composition of permutation
        such that $\phi_1\sigma=\phi_2$.
    \end{proof}

    Now, we give the abstract Definition \ref{defnoppairtoring} and then 
    show some Examples \ref{exmpsoppairringoper} to explain it.

    \begin{defn}\label{defnoppairtoring}
        Let $(\mathscr{C},\mathscr{G})$ be an operad pair. 
        Let $\tilde{\gamma}$ be the structure maps of both operads
        and $\lambda$ be the action of $\mathscr{G}$ on $\mathscr{C}$.
        The associated
        ring operad $\mathscr{R}_{\mathscr{C},\mathscr{G}}$ is defined as follows.
        
        (1) For any $f=\sum_{I=(i_1,i_2,\cdots,i_n)\in \{0,1\}^n} 
        \varepsilon_I a_{1,n}^{i_1}\cdots a_{n,n}^{i_n}\in \mathcal{R}(n)$,
        we define 
        $$\mathscr{R}_{\mathscr{C},\mathscr{G}}(f):=
        \mathscr{C}(|\Lambda_f|)\otimes\bigotimes_{I\in\Lambda_f}\mathscr{G}(|\Gamma_I|).$$

        (2) For any $(f_m,\phi,f_n)\in\widehat{\mathcal{R}}(f_m,f_n)$, say
            $$f_m=\sum_{I=(i_1,i_2,\cdots,i_m)\in \{0,1\}^m\setminus\{0\}^m} 
        \varepsilon_I a_{1,m}^{i_1}\cdots a_{m,m}^{i_m}\in \mathcal{R}(m).$$

        By Lemma \ref{type}, 
        $\phi$ induces an injection 
        $\phi': \Lambda_{f_n}\to \Lambda_{f_m}$
        sending each monomial summand $m\in\Lambda_{f_n}$
        to the unique monomial summand $m'\in\Lambda_{f_m}$
        such that $\phi_*m'=m$.

        Moreover, for each $J\in \Lambda_{f_n}$, 
        $\phi$ induces an injection 
        $\phi_{J}: \Gamma_J\to \Gamma_{\phi'J}$
        defined by the restriction of $\phi^{-1}$.

        Therefore, we define
        \begin{center}
            \begin{codi}
                \obj{
                    |(0)|\phi_*:\ \displaystyle\mathscr{C}(|\Lambda_{f_m}|)\otimes\bigotimes_{I\in\Lambda_{f_m}}\mathscr{G}(|\Gamma_I|)
                    &[15em] |(1)|\displaystyle\mathscr{C}(|\Lambda_{f_n}|)\otimes\bigotimes_{J\in\Lambda_{f_n}}\mathscr{G}(|\Gamma_J|)\\
                    };
                \mor 0 "{\phi'}^*\otimes \displaystyle\bigotimes_{J\in \Lambda_{f_n}}\phi_{J}^*":-> 1;
            \end{codi}
        \end{center}

        (3) The unit $\eta: * \to \mathscr{R}_{\mathscr{C},\mathscr{G}}(a_{1,1})=\mathscr{C}(1)\times\mathscr{G}(1)$
        is defined to be the product of two unit maps of $\mathscr{C}$ and $\mathscr{G}$.

        (4) 
        The composition map is defined as
        {\small\begin{center}
            \begin{codi}
                \obj{
                    |(0)|\displaystyle\mathscr{C}(|\Lambda_{f}|)\otimes\bigotimes_{I\in\Lambda_{f}}\mathscr{G}(|\Gamma_I|)
                    \otimes\bigotimes_{i=1}^{k}(\mathscr{C}(|\Lambda_{g_i}|)\otimes\bigotimes_{J\in\Lambda_{g_i}}\mathscr{G}(|\Gamma_J|))\\[0.5em]
                    |(2)|\displaystyle\mathscr{C}(|\Lambda_{f}|)\otimes\bigotimes_{I\in\Lambda_{f}}\bigg(\bigg(\mathscr{G}(|\Gamma_I|)
                    \otimes\bigotimes_{i\in \Gamma_I}\mathscr{C}(|\Lambda_{g_i}|)\bigg)\otimes
                    \mathscr{G}(|\Gamma_I|)\otimes\bigotimes_{i\in \Gamma_I}\bigotimes_{J\in\Lambda_{g_i}}\mathscr{G}(|\Gamma_J|)\bigg)\\[1em]
                    |(3)|\displaystyle\mathscr{C}(|\Lambda_{f}|)\otimes\bigotimes_{I\in\Lambda_{f}}\mathscr{C}(\prod_{i\in\Gamma_I}|\Lambda_{g_i}|)\otimes
                    \bigotimes_{I\in\Lambda_{f}}\bigotimes_{(J_1,\cdots,J_{|\Gamma_I|})\in\prod_{i\in\Gamma_I}
                    \Lambda_{g_i}}\bigg(\mathscr{G}(|\Gamma_I|)\otimes\bigotimes_{k=1}^{|\Gamma_I|}\mathscr{G}(|\Gamma_{J_k}|)\bigg)\\[2em]
                    |(5)|\displaystyle\mathscr{C}(\sum_{I\in\Lambda_f}\prod_{i\in\Gamma_I}|\Lambda_{g_i}|)\otimes\bigotimes_{I\in\Lambda_{f}}
                    \bigotimes_{(J_1,\cdots,J_{|\Gamma_I|})\in\prod_{i\in\Gamma_I}
                    \Lambda_{g_i}}\mathscr{G}(\sum_{k=1}^{|\Gamma_I|}|\Gamma_{J_k}|)\\[1em]
                    |(6)|\displaystyle\mathscr{C}(|\Lambda_{f(g_1,\cdots,g_k)}|)\otimes\bigotimes_{I\in\Lambda_{f(g_1,\cdots,g_k)}}\mathscr{G}(|\Gamma_I|)\\
                    };
                    \mor 0 "\Delta":-> 2 "\lambda":-> 3 "\tilde{\gamma}":-> 5 "\text{shuffle}":-> 6;
            \end{codi}
        \end{center}}
    \end{defn}

    \begin{prop}
        $\mathscr{R}_{\mathscr{C},\mathscr{G}}$ is a well-defined ring operad.
    \end{prop}
    \begin{proof}
        This is just a reformulation of the definition of an operad pair.
        $$\mathscr{R}_{\mathscr{C},\mathscr{G}}(0_n)=\mathscr{C}(0)=*$$
        by definition.
        The unit and equivariance diagrams commute from that of both 
        $\mathscr{C}$ and $\mathscr{G}$, and the associativity diagram
        commutes from the relation between $\lambda$ with the internal 
        structure $\tilde{\gamma}$ of both $\mathscr{C}$ and $\mathscr{G}$.
    \end{proof}

    \begin{exmps}\label{exmpsoppairringoper}
        (1) Consider a morphism 
        $$(f=a_{1,5}a_{2,5}a_{3,5}+a_{1,5}a_{4,5}+a_{5,5},\phi,g=a_{1,2}a_{2,2}+a_{1,2})$$
        where 
        \begin{align*}
            \phi:\{0,e,1,\cdots,5\}&\to \{0,e,1,2\},\\
            1&\mapsto e,\\
            2,4 &\mapsto 1,\\
            3&\mapsto 2,\\
            5&\mapsto 0.
        \end{align*}

        Then we have 
        \begin{align*}
            \mathscr{R}_{\mathscr{C},\mathscr{G}}(f)&=\mathscr{C}(3)\otimes \mathscr{G}(3)\otimes \mathscr{G}(2)\otimes \mathscr{G}(1),\\
            \mathscr{R}_{\mathscr{C},\mathscr{G}}(g)&=\mathscr{C}(2)\otimes\mathscr{G}(2)\otimes \mathscr{G}(1),
        \end{align*}
        together with 
        \begin{align*}
            \phi': \Lambda_{g}=\{\{1\}<\{1,2\}\}&\to \Lambda_{f}=\{\{5\}<\{1,4\}<\{1,2,3\}\}\\
            \{1\}&\mapsto\{1,4\}\\
            \{1,2\}&\mapsto\{1,2,3\}
        \end{align*}
        and 
        \begin{align*}
            \phi_{\{1\}}: \{1\}&\to\{1,4\}\\
            1&\mapsto4\\
            \phi_{\{1,2\}}: \{1,2\}&\to\{1,2,3\}\\
            1&\mapsto2\\
            2&\mapsto3.
        \end{align*}

        Therefore, $\phi_*:\mathscr{R}_{\mathscr{C},\mathscr{G}}(f)\to 
        \mathscr{R}_{\mathscr{C},\mathscr{G}}(g)$
        is the tensor of 
        \begin{align*}
            {\phi'}^*:\mathscr{C}(3)&\to \mathscr{C}(2),\\
            \phi_{\{1\}}^*: \mathscr{G}(2)&\to\mathscr{G}(1),\\
            \phi_{\{1,2\}}^*: \mathscr{G}(3)&\to\mathscr{G}(2),\\
            \mathscr{G}(1)&\to *.
        \end{align*}

    (2) Consider polynomials 
    \begin{align*}
        f&=a_{1,2}+a_{1,2}a_{2,2},\\
        g_1&=a_{1,2}+a_{1,2}a_{2,2},\\
        g_2&=a_{1,2}a_{2,2}.
    \end{align*}

    Then we have 
    $f(g_1,g_2)=a_{1,4}+a_{1,4}a_{3,4}a_{4,4}+a_{1,4}a_{2,4}+a_{1,4}a_{2,4}a_{3,4}a_{4,4}$
    and  
    \begin{align*}
        \mathscr{R}_{\mathscr{C},\mathscr{G}}(f)&=\mathscr{C}(2)\otimes \mathscr{G}(1)\otimes \mathscr{G}(2),\\
        \mathscr{R}_{\mathscr{C},\mathscr{G}}(g_1)&=\mathscr{C}(2)'\otimes\mathscr{G}(1)'\otimes \mathscr{G}(2)',\\
        \mathscr{R}_{\mathscr{C},\mathscr{G}}(g_2)&=\mathscr{C}(1)''\otimes\mathscr{G}(2)'',\\
        \mathscr{R}_{\mathscr{C},\mathscr{G}}(f(g_1,g_2))&=\mathscr{C}(4)\otimes\mathscr{G}(1)\otimes \mathscr{G}(3)\otimes\mathscr{G}(2)\otimes \mathscr{G}(4).
    \end{align*}

    In this case, 
    $ \gamma:\mathscr{R}_{\mathscr{C},\mathscr{G}}(f)\otimes
    \mathscr{R}_{\mathscr{C},\mathscr{G}}(g_1)\otimes
    \mathscr{R}_{\mathscr{C},\mathscr{G}}(g_2)\to 
    \mathscr{R}_{\mathscr{C},\mathscr{G}}(f(g_1,g_2))$ is gven by 
    the tensor product of 
    \begin{align*}
        \tilde{\gamma}:\mathscr{G}(1)\otimes\mathscr{G}(1)'&\to \mathscr{G}(1),\\
        \tilde{\gamma}:\mathscr{G}(2)\otimes\mathscr{G}(1)'\otimes\mathscr{G}(2)''&\to \mathscr{G}(3),\\
        \tilde{\gamma}:\mathscr{G}(1)\otimes\mathscr{G}(2)'&\to \mathscr{G}(2),\\
        \tilde{\gamma}:\mathscr{G}(2)\otimes\mathscr{G}(2)'\otimes\mathscr{G}(2)''&\to \mathscr{G}(4)
    \end{align*}
    and 
    \begin{center}
        \begin{codi}
            \obj{
                |(0)|\mathscr{C}(2)\otimes (\mathscr{G}(1)\otimes\mathscr{C}(2)')\otimes(\mathscr{G}(2)\otimes\mathscr{C}(2)'\otimes\mathscr{C}(1)'')\\
                |(1)|\mathscr{C}(2)\otimes (\mathscr{C}(2)\otimes\mathscr{C}(2))\\
                |(2)| \mathscr{C}(4)\\
                };
            \mor 0 "id\otimes\lambda\otimes\lambda":-> 1 "\tilde{\gamma}":-> 2;
        \end{codi}
    \end{center}
    \end{exmps}

    \begin{prop}\label{propringoperadandoperadpairsamealg}
        Let $(\mathscr{C},\mathscr{G})$ be an operad pair with associated
        ring operad $\mathscr{R}_{\mathscr{C},\mathscr{G}}$. Then
        the category of $(\mathscr{C},\mathscr{G})$-algebras 
        is isomorphic to the category of 
        $\mathscr{R}_{\mathscr{C},\mathscr{G}}$-algebras.
    \end{prop}
    \begin{proof}
        Let $(X,\theta)$ be an $\mathscr{R}_{\mathscr{C},\mathscr{G}}$-algebra.
        We define a $(\mathscr{C},\mathscr{G})$-algebra structure as follows.
        \begin{center}
            \begin{codi}
                \obj {|(0)|\theta_{+}: \mathscr{C}(j)\otimes X^{\otimes j}
                &[14em] |(1)| \mathscr{C}(j)\otimes\mathscr{G}(1)^j\otimes X^{\otimes j}
                =\mathscr{R}_{\mathscr{C},\mathscr{G}}(a_{1,j}+\cdots+a_{j,j})\otimes X^{\otimes j}
                &[9em] |(2)| X,\\[-2em]
                |(0')|\theta_{\times}: \mathscr{G}(j)\otimes X^{\otimes j}
                & |(1')| \mathscr{C}(1)\otimes\mathscr{G}(j)\otimes X^{\otimes j}
                =\mathscr{R}_{\mathscr{C},\mathscr{G}}(a_{1,j}\cdots a_{j,j})\otimes X^{\otimes j}
                & |(2')| X.\\
                };
                \mor 0 "\mathrm{id}\times\eta^j\times\mathrm{id}":-> 1 "\theta":-> 2;
                \mor 0' "\eta\times\mathrm{id}\times\mathrm{id}":-> 1' "\theta":-> 2';
            \end{codi}
        \end{center}

        Conversely, let $(Y,\theta_+,\theta_{\times})$ be a
        $(\mathscr{C},\mathscr{G})$-algebra. 
        We define a $\mathscr{R}_{\mathscr{C},\mathscr{G}}$-algebra 
        structure as follows.

        For any $$f=\sum_{I=(i_1,i_2,\cdots,i_n)\in \{0,1\}^n} 
        \varepsilon_I a_{1,n}^{i_1}\cdots a_{n,n}^{i_n}\in \mathcal{R}(n),$$
        assume $\Lambda_f=\{I_1<\cdots<I_{|\Lambda_f|}\}$ under the 
        lexicographical order, and assume
        $I_j=\{i_{1,j}<\cdots<i_{|\Gamma_{I_j}|,j}\}$.
        Then we define
        \begin{align*}
            \theta: \mathscr{R}_{\mathscr{C},\mathscr{G}}(f)\otimes X^{\otimes|f|}
                =\mathscr{C}(|\Lambda_f|)\otimes\bigotimes_{I\in\Lambda_f}\mathscr{G}(|\Gamma_I|)\otimes X^{\otimes|f|}\to X
        \end{align*}
        to be the composition 
        \begin{center}
            \begin{codi}
                \obj {|(0)|\mathscr{R}_{\mathscr{C},\mathscr{G}}(f)\otimes X^{\otimes|f|}
                =\mathscr{C}(|\Lambda_f|)\otimes\bigotimes_{I\in\Lambda_f}\mathscr{G}(|\Gamma_I|)\otimes 
                X_1\otimes\cdots\otimes X_{|f|}\\
                |(1)| \mathscr{C}(|\Lambda_f|)\otimes\bigotimes_{j=1}^{|\Lambda_f|}\biggl(\mathscr{G}(|\Gamma_{I_j}|)\otimes 
                X_{_{1,j}}\otimes\cdots\otimes X_{i_{|\Gamma_{I_j}|,j}}\biggr)\\
                |(2)| \mathscr{C}(|\Lambda_f|)\otimes X^{\otimes|\Lambda_f|}\\
                |(3)| X\\
                };
                \mor 0 "\text{shuffle}":-> 1 "\theta_{\times}":-> 2 "\theta_+":-> 3;
            \end{codi}
        \end{center}

        The above correspondence gives an isomorphism between 
        the category of $(\mathscr{C},\mathscr{G})$-algebras 
        and the category of 
        $\mathscr{R}_{\mathscr{C},\mathscr{G}}$-algebras.
    \end{proof}

    Now we assume $(\mathscr{C},\mathscr{G})$ is an $E_\infty$ operad pair
    in $\mathscr{U}$. 

    \begin{prop}\label{propoperadpairEinftyimpliesringoperadEinfty}
        Let $(\mathscr{C},\mathscr{G})$ be an $E_\infty$ operad pair.
        Then $\mathscr{R}_{\mathscr{C},\mathscr{G}}$ is an $E_\infty$
        ring operad.
    \end{prop}
    \begin{proof}
        First, $\mathscr{R}_{\mathscr{C},\mathscr{G}}(f)$ is contractible
        since it is a finite product of contractible spaces. Conditions 
        (2) and (5) in Definition \ref{defnEinfty} hold since 
        $\phi_*:\mathscr{R}_{\mathscr{C},\mathscr{G}}(f)\to\mathscr{R}_{\mathscr{C},\mathscr{G}}(g)$ 
        is a homeomorphism for all effective morphism $(f,\phi,g)$ in 
        $\widehat{\mathcal{R}}_{eff}$.
        Condition (4) follows from the freeness of both $\Sigma$ actions 
        on $\mathscr{C}$ and $\mathscr{G}$.

        For condition (3), given non-degenerate objects
        $f_1,f_2$ and $\alpha_1\in\mathscr{C}(f_1),\alpha_2\in\mathscr{C}(f_2)$,
        if there exists some $g$ with effective morphisms
        $(f_1,\phi_1,g), (f_2,\phi_2,g)$ in $\widehat{\mathcal{R}}_{eff}$
        such that $$\phi_{1*}\alpha_1=\phi_{2*}\alpha_2\in\mathscr{C}(g),$$
        then $f_1,f_2,g$ are connected.

        Let $h$ be the special object of the same type as $g$. Then      
        there must exist morphisms
        $(h,\psi_1,f_1), (h,\psi_2,f_2)$ in $\widehat{\mathcal{R}}_{n.d.}$
        such that $\phi_1\psi_1=\phi_2\psi_2$ by 
        Lemma \ref{lemofspecialpullback}.
        Since $\psi_{1*}$ is a homeomorphism,
        it follows that there exists $\beta\in\mathscr{C}(h)$
        such that $\psi_{1*}\beta=\alpha_1$.
        Now 
        \begin{align*}
            \phi_{2*}\psi_{2*}\beta=\phi_{1*}\psi_{1*}\beta=\phi_{1*}\alpha_1=\phi_{2*}\alpha_2
        \end{align*}
        implies that $\psi_{2*}\beta=\alpha_2$ since $\phi_{2*}$ is a homeomorphism,
        hence condition (3) follows.
    \end{proof}

    Therefore, applying the Comparison Theorem \ref{thmcompare1}
    and the classical multiplicative infinite loop machine as shown in \cite{may1982multiplicative},
    we get a multiplicative infinite loop machine defined 
    on algebras
    over any $E_\infty$ ring operad.

    \begin{thm}\label{thmgpcom}
        Let $(\mathscr{K},\mathscr{L})$ be the canonical 
        operad pair with associated monad pair 
        $(\mathbb{K},\mathbb{L})$ and associated ring operad
        $\mathscr{R}$.
        Let $\mathscr{C}$ be an arbitrary $E_\infty$ ring 
        operad. We denote the monad associated 
        to $\mathscr{C}\times\mathscr{R}$
        and $\mathscr{R}$
        by $\mathbb{D}$
        and $\mathbb{R}$, respectively.

        Then for any $X$ in $\mathscr{C}[\mathscr{U}_e]$, 
        the following composition is a group completion.
        $$X\simeq B(\mathbb{D},\mathbb{D},X) 
        \to B(\mathbb{R},\mathbb{D},X)\simeq B(\mathbb{K},\mathbb{K},B(\mathbb{R},\mathbb{D},X))
        \to\Omega^\infty B(\Sigma^\infty,\mathbb{K},B(\mathbb{R},\mathbb{D},X))
        $$

        Moreover, $B(\Sigma^\infty,\mathbb{K},B(\mathbb{R},\mathbb{D},X))$
        is an $E_\infty$ ring spectrum.
    \end{thm}

    \subsection{Ring operads and categories of ring operators}\label{Sectionringoperator}

    Historically, the construction of $E_\infty$ ring spaces
    from bipermutative categories is given in \cite{may1982multiplicative} and
    \cite{TheconstructionofE_infringspacesfrombipermutativecategories}
    in which an intermediate theory is used, namely, the theory of 
    categories of ring operators. We will give an alternative
    construction in Section \ref{Sectionsymbimonoidal} using 
    the theory of ring operads. Before that, we briefly describe a 
    comparison between ring operads
    and categories of ring operators. We first recall some notations defined
    in \cite{TheconstructionofE_infringspacesfrombipermutativecategories}.

    Let $\mathscr{F}$ be the category of 
    finite based sets $\mathbf{n}=\{0,1,2,\cdots,n\}$, with $0$ as
    basepoint, and based functions. Let $\Pi\subset\mathscr{F}$ be the subcategory
    whose morphisms are the based functions $\phi:\mathbf{m}\to\mathbf{n}$ such that 
    $|\phi^{-1}(j)|\leqslant1$ for
    $1 \leqslant j \leqslant n$, where $|S|$ denotes the cardinality of a 
    finite set $S$.

    \begin{defn}\label{defnoffwrf}
        Let $\varepsilon:\mathscr{K}\to \mathscr{F}$ and 
        $\mathscr{D}\to \mathscr{F}$ be two topological categories over $\mathscr{F}$
        which have the same objects as $\mathscr{F}$. Then 
        the objects of $\mathscr{K}\wr\mathscr{D} $
        are $n$-tuples of non-negative integers for all $n\geqslant 0$.
        We write such an object as $(n,S)=(n,s_1,\cdots,s_n)$.
        Moreover, morphisms are defined as 
        $$\mathscr{K}\wr\mathscr{D}((m,R),(n,S)):=
        \coprod_{\phi\in\mathscr{F}(\mathbf{m}_*,\mathbf{n}_*)}
        \varepsilon^{-1}(\phi)\times 
        \prod_{1\leqslant j \leqslant n} \mathscr{D}(\bigwedge_{\phi(i)=j}\mathbf{r}_{i*},\mathbf{s}_{j*})$$
        where the empty smash product is $\mathbf{1}_*$.
    \end{defn}

    \begin{defn}\label{defnofcatofringop}
        A category of ring operators is a topological category $\mathscr{J}$ 
        with objects those of $\Pi \wr \Pi$ such that the inclusion 
        $\Pi \wr \Pi \subset \mathscr{F}\wr\mathscr{F}$ factors as the composite of
        an inclusion $\Pi \wr \Pi \subset \mathscr{J}$ and a surjection 
        $\mathscr{J} \to \mathscr{F}\wr\mathscr{F}$ , both of which are the
        identity on objects. We require the maps 
        $\mathscr{J}((l,Q),(m,R))\to \mathscr{J}((l,Q),(n,S))$
        induced by an injection $(\phi,\chi):(m,R)\to (n,S)$  
        to be $\Sigma(\phi,\chi)$-cofibrations.

        Here $\Sigma(\phi,\chi)$ is the group of automorphisms 
        $(\sigma,\tau) :(n,S)\to (n,S)$
        such that $(\sigma,\tau)\mathrm{Im}(\phi,\chi)\subset \mathrm{Im}(\phi,\chi)$
        where $\mathrm{Im}(\phi,\chi)=\sqcup_i\mathrm{Im}\chi_i\subset\sqcup_i\mathbf{s_i}$.

        We denote the category of $\mathscr{J}$-spaces (functor category 
        from $\mathscr{J}$ to $\mathscr{U}$) by $\mathscr{J}[\mathscr{U}]$
        and the category of special $\mathscr{J}$-spaces 
        (see \cite[Definition 5.5]{TheconstructionofE_infringspacesfrombipermutativecategories})
        by $\mathscr{J}^s[\mathscr{U}]$
    \end{defn}

    Now we assign to each $E_\infty$ ring operad a 
    category of ring operators 
    $\tilde{\mathscr{C}}$. We give the abstract Definition 
    \ref{defnconsofringandcat} and then 
    show an Example \ref{examconsofringandcat} to explain it.

    \begin{defn}\label{defnconsofringandcat}
        Let $\mathscr{C}$ be an $E_\infty$ ring operad. 
        We define a category of ring operators 
        $\varepsilon:\tilde{\mathscr{C}}\to \mathscr{F}\wr\mathscr{F}$
        as follows.

        Let $(\phi,d)=(\phi,d_1,\cdots,d_n)\in\mathscr{F}((m,R),(n,S))$
        be any morphism in $\mathscr{F}$.
        To each pair $(h,j)$ with $1\leqslant j\leqslant n$,
        $1\leqslant h \leqslant s_j$, we assign a polynomial 
        $f_{\phi,d,h,j} $ in 
        $$\mathcal{R}(|R|)\cup\{1\}\subset \mathbb{Z}[a_{1,|R|},\cdots,a_{|R|,|R|}]$$
        as follow.
        Here $|R|:=\sum_i r_i$.

        (1) If $\phi^{-1}(j)\neq \emptyset$,
        then we assume $\phi^{-1}(j)=\{i_1<\cdots< i_l\}$
        and define 
        \begin{align*}
            f_{\phi,d,h,j}=
            \sum_{(k_1,\cdots,k_l)\in d_j^{-1}(h)}
            \prod_{t=1}^l
            a_{\sum_{s=1}^{i_t-1}r_s+k_1,|R|}.
        \end{align*}

        (2) If $\phi^{-1}(j)=\emptyset$, then 
        $$f_{\phi,d,h,j}:=\begin{cases}
            0_{|R|} &\text{ if } d_j^{-1}(h)=0\in\mathbf{1}_*,\\
            1_{|R|} &\text{ if } d_j^{-1}(h)=1\in\mathbf{1}_*.\\
        \end{cases}$$

        Then we define 
        $\varepsilon^{-1}(\phi,d):=\prod_{(h,j)}\mathscr{C}(f_{\phi,d,h,j}) $.
        Here $\mathscr{C}(1)$ consists of a single point.

        Note that for 
        \begin{align*}
            (\phi,d)=(\phi,d_1,\cdots,d_n)&\in\mathscr{F}\wr\mathscr{F}((m,R),(n,S)),\\
            (\phi',d')=(\phi',d'_1,\cdots,d'_n)&\in\mathscr{F}\wr\mathscr{F}((l,Q),(m,R)),
        \end{align*}
        the polynomial associated to their composition is 
        \begin{align*}
            &f_{(\phi,d)(\phi',d'),h,j}(a_{1,|Q|},\cdots,a_{|Q|,|Q|})\\
            =&f_{\phi,d,h,j}(f_{\phi',d',1,1},\cdots,
        f_{\phi',d',q_1,1},\cdots,f_{\phi',d',1,l},\cdots,f_{\phi',d',q_l,l})
        (a_{1,|Q|},\cdots,a_{|Q|,|Q|})
        \end{align*}
        
        The composition of morphisms in $\tilde{\mathscr{C}}$
        is defined to be the induced maps of $f_{\phi,d,h,j}$'s by composition
        of polynomials and evaluation of $1$'s.

        When $(\phi,d)=(\phi,d_1,\cdots,d_n)\in\mathscr{\Pi}\wr\mathscr{\Pi}((m,R),(n,S))$,
        all polynomials $f_{\phi,d,h,j}$'s are of the form 
        $0_{|R|}$, $1_{|R|}$, $a_{k,|R|}$,
        so $\varepsilon^{-1}(\phi,d)$ is a product of 
        $\mathscr{C}(0_{|R|})\cong*$, $\mathscr{C}(1_{|R|})\cong*$,
        $\mathscr{C}(a_{k,|R|})$. Therefore,
        there is a well-defined functor 
        $\Pi\wr\Pi\to \tilde{\mathscr{C}}$
        induced by the unit $\eta:*\to \mathscr{C}(a_{1,1})\to \mathscr{C}(a_{k,|R|})$.

        Moreover, when $(\phi,\chi):(m,R)\to (n,S)$ is an injection and 
        $\mathscr{C}$ is $E_\infty$, for any $(\psi,d):(l,Q)\to (m,R)$
        the induced map 
        $\varepsilon^{-1}(\psi,d)\to \varepsilon^{-1}((\phi,\chi)\circ(\psi,d))$
        is a permutation on components together with a 
        product with $*\to \mathscr{C}(0)$ and $*\to \mathscr{C}(1)$, so        
        the induced map  
        $\tilde{\mathscr{C}}((l,Q),(m,R))\to \tilde{\mathscr{C}}((l,Q),(n,S))$ is a 
        $\Sigma(\phi,\chi)$-cofibration.
    \end{defn}

    \begin{exmp}\label{examconsofringandcat}
        For example, consider 
        \begin{align*}
            (\phi,d)=(\phi,d_1)&\in\mathscr{F}\wr\mathscr{F}((2,(2,1)),(1,1)),\\
            (\phi',d')=(\phi',d'_1,d'_2)&\in\mathscr{F}\wr\mathscr{F}((2,(2,2)),(2,(2,1))),
        \end{align*}
        where 
        \begin{align*}
            \phi(1)=\phi(2)=1\\
            d_1(1,1)=d_1(2,1)=1\\
            \phi'(1)=\phi'(2)=1\\
            d'_1(1,1)=d'_1(2,2)=1\\
            d'_1(1,2)=2\\
            d'_1(2,1)=0\\
            d'_2(1)=1
        \end{align*}

        Then 
        \begin{align*}
            f_{\phi,d,1,1}&=a_{1,3}a_{3,3}+a_{2,3}a_{3,3}\\
            f_{\phi',d',1,1}&=a_{1,4}a_{3,4}+a_{2,4}a_{4,4}\\
            f_{\phi',d',2,1}&=a_{1,4}a_{4,4}\\
            f_{\phi',d',1,2}&=1_4
        \end{align*}
        and
        \begin{align*}
            f_{(\phi,d)(\phi',d'),1,1}=&a_{1,4}a_{3,4}+a_{2,4}a_{4,4}+a_{1,4}a_{4,4}
        \end{align*}
        while 
        \begin{align*}
            &f_{\phi,d,1,1}(f_{\phi',d',1,1},f_{\phi',d',2,1},a_{1,1})\\
            =&(a_{1,13}a_{3,13}+a_{2,13}a_{4,13})a_{13,13}+(a_{5,13}a_{8,13})a_{13,13}\\
            =&(a_{1,13}a_{3,13}+a_{2,13}a_{4,13}+a_{5,13}a_{8,13})a_{13,13}
        \end{align*}

        Therefore, the composition of morphisms in $\tilde{\mathscr{C}}$
        is given by
        \begin{center}
            \begin{codi}
                \obj {|(1)| \mathscr{C}(f_{\phi,d,1,1})\times\mathscr{C}(f_{\phi',d',1,1})
                \times \mathscr{C}(f_{\phi',d',2,1}) \times *\\[-1em]
                |(1')|\mathscr{C}(f_{\phi,d,1,1})\times\mathscr{C}(f_{\phi',d',1,1})
                \times \mathscr{C}(f_{\phi',d',2,1}) \times \mathscr{C}(a_{1,1}) \\[-1em]
                |(2)|\mathscr{C}((a_{1,13}a_{3,13}+a_{2,13}a_{4,13}+a_{5,13}a_{8,13})a_{13,13})\\[-1em]
                |(3)|\mathscr{C}(a_{1,4}a_{3,4}+a_{2,4}a_{4,4}+a_{1,4}a_{4,4})\\
                };
                \mor 1 "\eta":-> 1' "\gamma":-> 2 "\psi_*":-> 3;
            \end{codi}
        \end{center}
        where 
        \begin{align*}
            \psi:\{0,e,1,2,\cdots,12,13\}&\to \{0,e,1,2,\cdots,4\}\\
                n &\mapsto [n\ (\mathrm{mod}\ 4)] \text{ for } n=1,2,\cdots,12\\
                13&\mapsto e.
        \end{align*}
        Here we require $[n\ (\mathrm{mod}\ 4)]\in \{1,2,3,4\}$.
    \end{exmp}

    With this definition, we can construct a special $\tilde{\mathscr{C}}$
    space from any $\mathscr{C}$ algebra in $\mathscr{U}_e$.

    \begin{defn}\label{defnringandcatonalg}
        Let $\mathscr{C}$ be an $E_\infty$ ring operad with the 
        associated category of ring operad $\tilde{\mathscr{C}}$.
        Then there is a canonical functor $
        \nu:\mathscr{C}[\mathscr{U}_e]\to \tilde{\mathscr{C}}^s[\mathscr{U}]$
        defined as follows.

        Let $(X,\theta)$ be an object in $\mathscr{C}[\mathscr{U}_e]$.
        Then we define 
        \begin{align*}
            \nu X:\tilde{\mathscr{C}} &\to \mathscr{U}\\
            (n,S)&\mapsto X^{s_1}\times\cdots\times X^{s_n}.
        \end{align*}

        Moreover, for any $(\phi,d)=(\phi,d_1,\cdots,d_n)\in\mathscr{F}((m,R),(n,S))$
        and $(\alpha_{h,j})\in\varepsilon^{-1}(\phi,d):=\prod_{(h,j)}\mathscr{C}(f_{\phi,d,h,j}) $,
        we define
        \begin{align*}
            \nu X(\alpha_{h,j}):X^{r_1}\times\cdots\times X^{r_m} &\to X^{s_1}\times\cdots\times X^{s_n}\\
            (x_{1,1},\cdots,x_{r_1,1},\cdots,x_{1,m},\cdots,x_{r_m,m})
            &\mapsto (y_{1,1},\cdots,y_{s_1,1},\cdots,y_{1,n},\cdots,y_{s_n,n})
        \end{align*}
        where 
        $$y_{h,j}=\theta (\alpha_{h,j},x_{1,1},\cdots,x_{r_1,1},\cdots,x_{1,m},\cdots,x_{r_m,m})$$
        for $f_{\phi,d,h,j}\neq 1$ and $y_{h,j}=e$ otherwise.
    \end{defn}

    \begin{thm}
        Let $\mathscr{C}$ be an $E_\infty$ ring operad with 
        the associated category of ring operators $\tilde{\mathscr{C}}$.
        Then $
        \nu:\mathscr{C}[\mathscr{U}_e]\to \tilde{\mathscr{C}}^s[\mathscr{U}]$
        induces an equivalence on homotopy categories.
    \end{thm}

    \begin{proof}
        Let $(\mathscr{K},\mathscr{L})$ be the canonical
        operad pair with the associated ring operad $\mathscr{R}_{\mathscr{K},\mathscr{L}}$.
        Let $\mathscr{D}:=\mathscr{C}\times \mathscr{R}_{\mathscr{K},\mathscr{L}}$.

        Consider the following diagram
        \begin{center}
            \begin{codi}
                \obj {|(1)| \mathscr{C}[\mathscr{U}_e]
                &[9em] |(2)| \tilde{\mathscr{C}}^s[\mathscr{U}]\\[-1em]
                |(3)| \mathscr{D}[\mathscr{U}_e]
                &|(4)| \tilde{\mathscr{D}}^s[\mathscr{U}]\\[-1em]
                |(5)|  \mathscr{R}_{\mathscr{K},\mathscr{L}}[\mathscr{U}_e]
                =(\mathscr{K},\mathscr{L})[\mathscr{U}_e]
                &|(6)| \tilde{\mathscr{R}}^s_{\mathscr{K},\mathscr{L}}[\mathscr{U}]
                =(\hat{\mathscr{L}}\wr \hat{\mathscr{K}})^s[\mathscr{U}],\\
                };
                \mor 1 "\nu":-> 2;
                \mor 3 "\nu":-> 4;
                \mor 5 "\nu":-> 6;
                \mor 1 -> 3;
                \mor 2 -> 4;
                \mor 5 -> 3;
                \mor 6 -> 4;
            \end{codi}
        \end{center}

        Here the above diagram commutes by Definition \ref{defnringandcatonalg}.
        Comparing with Definition \ref{defnoppairtoring}, Definition \ref{defnconsofringandcat} and 
        \cite[Definition 1.2, Definition 5.1]{TheconstructionofE_infringspacesfrombipermutativecategories},
        we get $$\tilde{\mathscr{R}}_{\mathscr{K},\mathscr{L}}
                =(\hat{\mathscr{L}}\wr \hat{\mathscr{K}}).$$

        Moreover, comparing Definition \ref{defnringandcatonalg} and 
        \cite[Definition 6.1]{TheconstructionofE_infringspacesfrombipermutativecategories},
        the bottom horizontal functor $\nu$ coincides with $R=R''R'$
        defined in \cite{TheconstructionofE_infringspacesfrombipermutativecategories}.

        All vertical functors in the above diagram induce 
        equivalences on homotopy categories by Theorem \ref{thmcompare1}
        and \cite[Theorem 5.11]{TheconstructionofE_infringspacesfrombipermutativecategories}.
        The bottom horizontal functor
        induces an
        equivalence on homotopy categories
        by \cite[Theorem 8.6, Theorem 10.6]{TheconstructionofE_infringspacesfrombipermutativecategories}.

        Therefore, all functors in the above diagram induce 
        equivalences on homotopy categories.
    \end{proof}

    \begin{rem}
        In the proof of the above theorem, we reduced it to the special
        case when $\mathscr{C}$ is the ring operad associated to 
        some $E_\infty$ operad pair. The special case is also not 
        easy to prove but it has been proved 
        in \cite{TheconstructionofE_infringspacesfrombipermutativecategories}.
        However, the proof of the special case cannot 
        (at least not in an obvious way) be generalized to prove 
        the theorem because the monad in $\Pi\wr\Pi$-spaces associated
        to $\tilde{\mathscr{C}}$ for a general $E_\infty$ ring operad 
        $\mathscr{C}$ is hard to describe.
        Only the monad associated to $\hat{\mathscr{L}}\wr \hat{\mathscr{K}}$
        for some $E_\infty$ operad pair $(\mathscr{K},\mathscr{L})$
        has been described in \cite{TheconstructionofE_infringspacesfrombipermutativecategories}.
    \end{rem}

    \section{Applications in category theory}\label{sectioncategory}

    \subsection{Ring operad for symmetric bimonoidal categories}\label{Sectionsymbimonoidal}

    As an application, we show that the 
    classifying space of a symmetric bimonoidal category
    is equivalent to some $(\mathscr{K},\mathscr{L})$-algebra,
    where $(\mathscr{K},\mathscr{L})$ is the canonical operad pair.
    This is originally proved in 
    \cite{TheconstructionofE_infringspacesfrombipermutativecategories}.
    
    In \cite{OperadsforSymmetricMonoidalCategories}, Elmendorf
    constructs an operad in the category of small categories
    whose algebras are precisely symmetric monoidal 
    categories. We modify this construction to get a ring operad $\mathscr{S}$
    in $(\mathbf{Cat},\times,*)$ 
    such that $\mathscr{S}$-algebras are precisely tight symmetric bimonoidal 
    categories (with strict zero object and unit object).
    This ring operad cannot be induced by any operad pair, and it 
    shows the difference between operad pairs and ring operads. 
    
    Originally, the coherence theorem for 
    symmetric bimonoidal 
    categories was first proved by Laplaza in \cite{laplaza2006coherence},
    in which some details are omitted. A complete proof is given in 
    \cite{johnson2021bimonoidal}, and our notations in this section 
    are also due to \cite{johnson2021bimonoidal}.

    \begin{notn}
        Let $E:(\mathbf{Set},\times,*)\to (\mathbf{Cat},\times,*)$
        be the functor sending a set $X$ to its 
        indiscrete category; that is, objects in $EX$ are elements
        in $X$ and each hom-set consists of exactly one element.
    \end{notn}

    Here are some propositions about this functor $E$. All of them can be 
    easily checked by definition.

    \begin{prop}\label{propofE}
        (1) $E$ is fully faithful.

        (2) $E$ is symmetric monoidal.

        (3) The nerve of $EX$ is precisely the free simplicial set generated
        by vertexes $X$. Here the free simplicial set functor is the left 
        adjoint of $X_*\mapsto X_0$.

        (4) The classifying space $BEX$ of $EX$ is always contractible.

        (5) If $X\to X'$ is an injection between sets, then the 
        induced map $BEX\to BEX'$ is a cofibration.
    \end{prop}

    Therefore, if $\mathscr{C}$ is a ring operad in $(\mathbf{Set},\times,*)$,
    then $E\mathscr{C}$ is a ring operad in $(\mathbf{Cat},\times,*)$.
    We first construct a ring operad $\mathscr{S}_{set}$ in $(\mathbf{Set},\times,*)$
    as follows.

    As defined in \cite{johnson2021bimonoidal},
    a $\{+,\times\}$-algebra is a set with two specific elements $0,1$
    and two binary operations $+,\times$.
    Let $A_n$ be  
    the free $\{+,\times\}$-algebra
    generated by $ \{0_n,1_n,a_{1,n},\cdots,a_{n,n}\} $ quotient out relations 
    \begin{align*}
        0_n+x&=x+0_n=x,\\
        0_n\times x&=x\times 0_n=0_n,\\
        1_n\times x&=x\times 1_n=x,\\
    \end{align*}

    Therefore, there is a canonical $\{+,\times\}$-algebra morphism
    \begin{align*}
        p_n: A_n &\to \mathbb{Z}_{\geqslant0}[a_{1,n},\cdots,a_{n,n}]\\
        a_{k,n} &\mapsto a_{k,n},\\
        1_n&\mapsto 1_n,\\
        0_n&\mapsto 0_n.
    \end{align*}
    Moreover, for any $\phi:\mathbf{m}_e\to \mathbf{n}_e$,
    there is a induced $\{+,\times\}$-algebra morphism
    \begin{align*}
        \phi_*: A_m &\to A_n\\
        a_{k,m} &\mapsto a_{\phi(k),n}
    \end{align*}
    with $a_{0,n}=0$ and $a_{e,n}=1$ 
    such that the following diagram commutes:
    \begin{center}
        \begin{codi}
            \obj{
                |(0)|A_m
                &[5em]|(1)|A_n\\
                |(2)|\mathbb{Z}_{\geqslant 0}[a_{1,m},\cdots,a_{m,m}]
                &|(3)|\mathbb{Z}_{\geqslant 0}[a_{1,n},\cdots,a_{n,n}]\\
                };
            \mor 0 "\phi_*":-> 1 "p_n":-> 3;
            \mor 0 "p_m":-> 2 "\phi_*":-> 3;
        \end{codi}
    \end{center}

    Less formally, the ring operad $\mathscr{S}_{set}$ is defined as 
    the preimage of $\coprod \mathcal{R}(n)$ under $\coprod p_n$.
    \begin{defn}
        We define $\mathscr{S}_{set}$ as follows:

        (1) For any $f\in \mathcal{R}(n)$,
        $$\mathscr{S}_{set}(f):=p_n^{-1}(f);$$

        (2) For any $(f_m,\phi,f_n)\in \widehat{\mathcal{R}}(f_m,f_n)$,
        $\phi_*: \mathscr{S}_{set}(f_m)\to \mathscr{S}_{set}(f_n)$ is the restriction
        of $\phi_*$ on $\mathscr{S}_{set}(f_m)$;

        (3) The unit element is $a_{1,1}\in \mathscr{S}_{set}(a_{1,1})=p_1^{-1}(a_{1,1})$;

        (4) The composition map 
        $$\gamma:\mathscr{S}_{set}(f)\times\mathscr{S}_{set}(g_1)\times\cdots\times \mathscr{S}_{set}(g_k)
        \to \mathscr{S}_{set}(f(g_1,\cdots,g_k)) $$
        is induced by the composition of elements in $\coprod A_n$.
    \end{defn}

    It's easy to check $\mathscr{S}_{set}$ above is a well-defined ring operad,
    and we let $\mathscr{S}$ to be the ring operad in $(\mathbf{Cat},\times,*)$
    defined by $E\mathscr{S}_{set}$.

    \begin{thm}\label{thmsymbi}
        The algebras over $\mathscr{S}$ are precisely tight symmetric bimonoidal 
        categories (defined in 
        \cite[Volume I, Definition 2.1.2]{johnson2021bimonoidal}) with strict 
        zero and unit objects.
    \end{thm}

    \begin{proof}

        Note that the polynomial rig over non-negative integers 
        $\mathbb{Z}_{\geqslant 0}[a_{1,n},\cdots,a_{n,n}]$
        is by definition the quotient of $A_n$
        by relation given by associativity, commutativity, distributivity, and unit laws.
        So morphisms in $\mathscr{S}$
        are generated by morphisms of one of the following forms
        and their inverses
        \begin{align*}
            \alpha_{A,B,C} &: A\times(B\times C)\to (A\times B)\times C,\\
            \alpha'_{A,B,C} &: A+(B+ C)\to (A+B)+ C,\\
            \gamma_{A,B} &: A\times B\to B\times A,\\
            \gamma'_{A,B} &: A+ B\to B+ A,\\
            \lambda_{A,n}&: 1_n\times A\to A,\\
            \rho_{A,n}&: A\times 1_n\to A,\\
            \delta_{A,B,C}&: A\times(B+C)\to A\times B+A\times C,\\
            \delta^{\#}_{A,B,C}&: (A+B)\times C\to A\times C+B\times C.
        \end{align*}

        Therefore, if $C$ is a $\mathscr{S}$-algebra with structure map
        $\lambda: \mathscr{S}(f)\to \mathrm{Func}(C^n,C)$,
        then 
        $$(C,\gamma(a_{1,2}+a_{2,2}),\gamma(a_{1,2}+a_{2,2}))$$
        gives a tight symmetric bimonoidal 
        category (with strict zero object) structure on $C$.
        Here all the coherence diagrams commute since there is precisely one morphism
        in each hom-set in $\mathscr{S}$.

        Conversely, let $\tilde{\mathscr{S}}(f)$ be the free category generated by morphisms
        of one of the forms in the above list and their inverse 
        except for the inverses of $\delta_{A,B,C}$
        and $\delta^{\#}_{A,B,C}$. 
        If $(C,\oplus,\otimes,0,1)$ is a tight symmetric bimonoidal 
        category with strict zero object, then we can define a functor
        $\coprod_{|f|=n}\tilde{\mathscr{S}}(f)\to \mathrm{Func}(C^{n}, C)$ sending
        \begin{align*}
            a_{i,n}&\mapsto \{(x_1,\cdots,x_n)\mapsto x_i\} \\
            \alpha+\beta &\mapsto \{ (x_1,\cdots,x_n)\mapsto \alpha(x_1,\cdots,x_n)\oplus\beta(x_1,\cdots,x_n) \}\\
            \alpha\times\beta &\mapsto \{ (x_1,\cdots,x_n)\mapsto \alpha(x_1,\cdots,x_n)\otimes\beta(x_1,\cdots,x_n) \}\\
            0_n&\mapsto \{(x_1,\cdots,x_n)\mapsto 0\}\\
            1_n&\mapsto \{(x_1,\cdots,x_n)\mapsto 1\}
        \end{align*}
        and sending the morphisms listed above to the structure maps of 
        $\mathrm{Func}(C^{n}, C)$. Here the 
        tight symmetric bimonoidal of $\mathrm{Func}(C^{n}, C)$
        is induced by that of $C$.

        Comparing our Notation \ref{rn} and 
        \cite[Volume I, Definition 3.1.25]{johnson2021bimonoidal},
        we get an element $x\in A_n$ is regular in the sense of \cite{johnson2021bimonoidal}
        if and only if $p_n(x)\in\mathcal{R}(n)$.

        Then by the coherence theorem 
        \cite[Volume I, Theorem 3.9.1]{johnson2021bimonoidal},
        the above functor 
        $\tilde{\mathscr{S}}(f)\to \mathrm{Func}(C^{n}, C)$
        factors through the image of
        $\tilde{\mathscr{S}}(f)$
        in $\mathscr{S}(f)$, denoted by $\mathscr{S}'(f)$. 
        
        To define an $\mathscr{S}$ action on $C$, it suffices to extend the 
        above functor defined on $\mathscr{S}'(f)$ to $\mathscr{S}(f)$.
        Note that any morphism in $\mathscr{S}'(f)$ is sent to 
        an isomorphism in $\mathrm{Func}(C^{n}, C)$, so we only need 
        to show that $\mathscr{S}(f)$ is the free groupoid generated by 
        $\mathscr{S}'(f)$.

        Note that each hom-set in $\mathscr{S}'(f)$ contains exactly zero 
        or one element, so $\mathscr{S}'(f)$ is equivalent to some poset.
        Also, by \cite[Proposition 1]{jardine2020persistenthomotopytheory},
        the free groupoid of some poset $P$ is equivalent to the fundamental
        group of the classifying space $BP$. Therefore, it remains to show
        that $B\mathscr{S}'(f)$ is contractible.

        Indeed, $B\mathscr{S}'(f)$ is contractible because 
        $\mathscr{S}'(f)$ has a terminal object. 
        For 
        $$f=\sum_{k=1}^{l}a_{i_{1,k},n}a_{i_{2,k},n}\cdots a_{i_{j_k,k},n},$$
        the following element
        $$\biggl(\biggl(((a_{i_{1,1},n}a_{i_{2,1},n})\cdots a_{i_{j_1,1},n})+
        ((a_{i_{1,2},n}a_{i_{2,2},n})\cdots a_{i_{j_2,2},n})\biggr)+\cdots+
        ((a_{i_{1,l},n}a_{i_{2,l},n})\cdots a_{i_{j_l,l},n})\biggr)$$
        is terminal in $\mathscr{S}'(f)$, so the theorem holds.
    \end{proof}

    Note that the classifying space functor $B$ is product preserving, so the classifying
    space of $\mathscr{S}$ gives a ring operad in $\mathscr{U}$.

    \begin{prop}\label{propsymbiEinfty}
        $B\mathscr{S}$ is an $E_\infty$ ring operad.
    \end{prop}
    \begin{proof}
        First, all $B\mathscr{S}(f)$ are contractible by Proposition \ref{propofE}.

        To check conditions (2), (3) and (4), 
        note that if they hold for $\mathscr{S}_{set}$, 
        they also hold after taking the free simplicial sets generated by
        $\mathscr{S}_{set}$, and therefore hold after taking geometric realization.
        So it suffices to check them 
        on $\mathscr{S}_{set}$. 

        By Proposition \ref{propofE}, to check condition (5), it suffices
        to check all $\phi_*$ are injective on $\mathscr{S}_{set}$.

        Note that an element in $A_n$ is one-to-one corresponding to 
        a sequence with length $l$ of variables in $\{a_{1,n},\cdots,a_{n,n},1_n\}$,
        a sequence with length $l-1$ of operators $\{+,\times\}$,
        together with a parenthesization such that 
        this expression is not of the form 
        $$\cdots (1_n \times (\alpha))\cdots$$
        or 
        $$\cdots ((\alpha)\times 1_n )\cdots.$$
        Here $0_n$ does not appear because it has been cancelled by 
        strict nullity.

        Moreover, if $\phi_*a=b$ for two elements $a,b$ in $\coprod A_n$
        with $\phi$ effective,
        then $a$ and $b$ share the same sequence of operators and the same
        parenthesization. Only the sequences of variables
        are changed by $\phi$, but the length of the sequences of variables
        and the $1$'s in these sequences are also the same.

        With this prospective, we introduce the following notation.

        \begin{notn}
            For any element $\alpha\in A_n$,
            let 
            $$(a_{i_1,n},a_{i_2,n},\cdots,a_{i_l,n})$$
            be the sequence of variables associated to $\alpha$.
            Here $i_j\in\{e,1,2,\cdots,n\}$ for $j=1,2,\cdots,l$ 
            and $a_{e,n}=1_n$. We denote
            \begin{align*}
                \chi(j)=\begin{cases}
                    e &\text{ if } i_j=e,\\
                    j &\text{ if } i_j\neq e,\\
                \end{cases}
            \end{align*}
            for $j=1,2,\cdots,l$.

            Let $\tilde{\alpha}$ be the element in $A_l$
            with the same sequence of operators
            and parenthesization as $\alpha$ and with 
            the corresponding
            sequence of variables
            $$(a_{\chi(1),l},a_{\chi(2),l},\cdots,a_{\chi(l),l}).$$

            We say $1\leqslant k_1<k_2<\cdots<k_r\leqslant l$
            is a fundamental sequence of $\alpha$ if and only if:

            (1) $i_{k_j}\neq e$, i.e. $\chi(k_j)=k_j$, for $j=1,2,\cdots,r$;

            (2) in the monomial decomposition of $p_l(\tilde{\alpha})\in
            \mathbb{Z}_{\geqslant0}[a_{1,l},\cdots,a_{l,l}]$,
            the coefficient of the monomial $a_{k_1,l}a_{k_2,l}\cdots a_{k_r,l}$
            is positive.
        \end{notn}

        Now we return to the proof of Proposition \ref{propsymbiEinfty}. 
        Note that $A_n$ is only a $\{+,\times\}$-algebra and there 
        is no minus operations. So if $a_{m,n}$ does not appear in some 
        $f\in\mathcal{R}(n)$, $a_{m,n}$ cannot appear in the expression 
        of any element in $\mathscr{S}_{set}(f)$. Therefore, condition (2) holds.

        For condition (3), let $f_1,f_2$ be non-degenerate objects
        and 
        $$\alpha_1\in\mathscr{S}_{set}(f_1),\alpha_2\in\mathscr{S}_{set}(f_2).$$
        Suppose there exists some non-degenerate $g$ with effective morphisms
        $(f_1,\phi_1,g), (f_2,\phi_2,g)$ in $\widehat{\mathcal{R}}_{eff}$
        such that $$\phi_{1*}\alpha_1=\phi_{2*}\alpha_2\in\mathscr{S}_{set}(g).$$

        With the argument above, we regard $\alpha_1,\alpha_2$
        as sequences $(a_{i_1,|f_1|},\cdots,a_{i_l,|f_1|})$,
        $(a_{i'_1,|f_2|},\cdots,a_{i'_l,|f_l|})$, respectively.

        Therefore, $\phi_{1*}\alpha_1=\phi_{2*}\alpha_2$
        implies $\phi_{1}(i_k)=\phi_{2}(i'_k)$ for $k=1,2,\cdots, l$.

        Consider the following pull-back diagram in the category of 
        based finite sets with based point $e$.
        \begin{center}
            \begin{codi}
                \obj{
                    |(0)|\{e,1,\cdots, m\}
                    &[5em]|(1)|\{e,1,\cdots, |f_2|\}\\
                    |(2)|\{e,1,\cdots, |f_1|\}
                    &|(3)|\{e,1,\cdots, |g|\}\\
                    };
                \mor 0 "\psi_2":-> 1 "\phi_2":-> 3;
                \mor 0 "\psi_1":-> 2 "\phi_1":-> 3;
            \end{codi}
        \end{center}
        Let $\beta$ the element in $A_m$ corresponding to the sequence
        $$(a_{j_1,m},\cdots,a_{j_l,m})$$
        where $j_k$ is the unique element in $\{e,1,\cdots, m\}$
        such that $(\psi_1(j_k),\psi_2(j_k))=(i_k,i'_k)$.
        Then we get 
        \begin{align*}
            \psi_{1*}\beta&=\alpha_1,\\
            \psi_{2*}\beta&=\alpha_2.
        \end{align*}
        Here the projection $p_m(\beta)\in\mathbb{Z}_{\geqslant0}[a_{1,m},\cdots,a_{m,m}]$
        must be contained in $\mathcal{R}(m)$ since otherwise 
        $f_1=\psi_{1*}p_m(\beta)=p_m\psi_{1*}(\beta)$ is not contained
        in $\mathcal{R}(|f_1|)$. So condition (3) holds.
        
        For condition (4), let 
        $$(f,\phi_1,g),(f,\phi_2,g)$$
        be effective morphisms with $f$ non-degenerate, and 
        $\alpha\in\mathscr{C}(f)$ such that 
        $\phi_{1*}\alpha=\phi_{2*}\alpha$.
        We also regard $\alpha$ as a sequence of variables
        $$(a_{i_1,|f|},\cdots,a_{i_l,|f|}).$$
        Now $\phi_{1*}\alpha=\phi_{2*}\alpha$ implies 
        $\phi_{1}(i_k)=\phi_{2}(i_k)$, and $f$ non-degenerate
        implies $$\{e\}\cup\{i_1,\cdots,i_l\}=\{e,1,2,\cdots,|f|\},$$
        so condition (4) holds.
        
        For condition (5), let $(f,\phi,g)$ be a morphism
        in $\widehat{\mathcal{R}}_{n.d.}$.
        If there exists $\alpha_1,\alpha_2$ in $\mathscr{P}_s(f)$
        such that $\phi_*\alpha_1=\phi_*\alpha_2$,
        then we denote the corresponding
        sequences of operators with $\alpha_1,\alpha_2$ by 
        $(a_{i_1,|f|},\cdots,a_{i_l,|f|})$,
        $(a_{i'_1,|f|},\cdots,a_{i'_l,|f|})$, respectively.
        $\phi_*\alpha_1=\phi_*\alpha_2$ implies
        $\phi_*(i_k)=\phi_*(i'_k)$.

        We also have the projections $p_{|f|}(\alpha_1)=p_{|f|}(\alpha_2)=f$.
        Therefore, for a monomial 
        $a_{s_1,|f|}\cdots a_{s_r,|f|}$ which has coefficient $1$ 
        in the monomial decomposition of $f$, there must exist
        a fundamental sequence 
        $\{k_1<\cdots< k_r\}$ of $\alpha_1$ such that $(i_{k_1},\cdots,i_{k_r})$ is a permutation
        of $(s_1,\cdots,s_r)$. Moreover, 
        $\phi_{*}f=g\in\mathcal{R}(|g|)$ implies 
        $\phi(s_1),\cdots,\phi(s_r)$ are distinct.

        Similarly, there is a fundamental sequence $\{k'_1<\cdots< k'_r\}$
        of $\alpha_2$
        such that $(i'_{k'_1},\cdots,i'_{k'_r})$ is a permutation
        of $(s_1,\cdots,s_r)$. If $\{k'_1<\cdots< k'_r\}\neq \{k_1<\cdots< k_r\}$,
        then applying $p_{|g|}\phi_*\alpha_1=p_{|g|}\phi_*\alpha_2$, we get
        two same monomials in the summation of $g$, 
        and thus the coefficient of monomial 
        $a_{\phi(s_1),|g|}a_{\phi(s_2),|g|}\cdots a_{\phi(s_r),|g|}$
        is at least 2, which gives a contradiction.

        Therefore, $\{k'_1<\cdots< k'_r\}=\{k_1<\cdots< k_r\}$.
        Also, since $\phi(s_1),\cdots,\phi(s_r)$ are distinct,
        $\phi(i_{k_t})=\phi(i'_{k_t})\in \{\phi(s_1),\cdots,\phi(s_r)\}$
        together with $i_{k_t},i'_{k_t}\in\{s_1,\cdots,s_r\}$
        implies $i_{k_t}=i'_{k_t}$ for $t=1,2,\cdots,r$.

        Moreover, each $k=1,2,\cdots, l$ is either contained 
        in some fundamental sequence of $\alpha_1$
        or $i_k=i'_k=e$, so $i_{k}=i'_{k}$
        holds for all $k$. Therefore, $\alpha_1=\alpha_2$,
        so condition (5) holds.
    \end{proof}

    Therefore, applying the group completion theorem \ref{thmgpcom},
    we get an alternative proof of the following result.

    \begin{thm}\label{thmsynbigpcom}
        Let $C$ be a tight symmetric bimonoidal 
        category (with strict zero and unit object) and let $0$ 
        be the based-point of the classifying space $BC$.
        Then there is a group completion
        \begin{align*}
            BC\to \Omega^\infty \mathbb{E}(BC)
        \end{align*}
        where $\mathbb{E}(BC)$ is an $E_\infty$ ring spectrum  
        depending functorially on $C$.
    \end{thm}

    In particular, if we begin with  a 
    skeleton of the category of free
    $R$-modules with operators $\{\oplus,\otimes\}$ over some
    commutative ring $R$,
    then we get an alternative approach to construct the algebraic $K$-theory 
    ring spectrum.

    \subsection{From symmetric bimonoidal categories to bipermutative categories}\label{Sectionbipermutative}
    
    The classical construction of algebraic $K$-theory 
    ring spectrum \cite{TheconstructionofE_infringspacesfrombipermutativecategories}
    only concerns bipermutative categories, while we focus on general
    symmetric bimonoidal categories since the category of projective
    $R$-modules is only symmetric bimonoidal. The classical
    construction works due to a strictification functor
    sending each symmetric bimonoidal category
    to some bipermutative category with equivalent classifying
    space, see \cite[Section VI.3]{may1977ring}.

    In this section, we modify the construction in Section 
    \ref{Sectionsymbimonoidal} to get a ring operad for 
    bipermutative categories and define an operadic strictification
    functor. Then we use this to description a comparison 
    between the the two construction from bipermutative categories
    to classical $E_\infty$ ring spaces.

    Recall that a bipermutative category is a tight symmetric bimonoidal 
    category with strict unit and associativity for both 
    addition and multiplication and also a strict right distributivity.
    Hence, we define $A'_n$ to be the $\{+,\times\}$-algebra
    generated by $\{a_{1,n},\cdots,a_{n,n},1_n,0_n\}$ with relations 
    \begin{align*}
        0_n+x&=x+0_n=x,\\
        0_n\times x&=x\times 0_n=0_n,\\
        1_n\times x&=x\times 1_n=x,\\
        (x+y)+z&=x+(y+z),\\
        (x\times y)\times z&=x\times (y\times z),\\
        (x+y)\times z&=x\times z+y\times z.
    \end{align*}
    Therefore, the projection map $p_n: A_n\to \mathbb{Z}[a_{1,n},\cdots,a_{n,n}]$
    factor through $A'_n$ uniquely
    \begin{center}
        \begin{codi}
            \obj{
                |(0)|A_n
                &[5em] |(1)|\mathbb{Z}[a_{1,n},\cdots,a_{n,n}]\\
                |(2)|A'_n
                &\\
                };
            \mor 0 "p_n":-> 1;
            \mor 0 "\nu":-> 2 "p'_n":-> 1;
        \end{codi}
    \end{center}

    In analogy with Section \ref{Sectionsymbimonoidal},
    we define the ring operad $\mathscr{P}_{set}$ as 
    the preimage of $\coprod \mathcal{R}(n)$ under $\coprod p'_n$.
    \begin{defn}
        We define $\mathscr{P}_{set}$ as follows:

        (1) For any $f\in \mathcal{R}(n)$,
        $$\mathscr{P}_{set}(f):={p'}_n^{-1}(f);$$

        (2) For any $(f_m,\phi,f_n)\in \widehat{\mathcal{R}}(f_m,f_n)$,
        $\phi_*: \mathscr{P}_{set}(f_m)\to \mathscr{P}_{set}(f_n)$ is the restriction
        of $\phi_*$ on $\mathscr{P}_{set}(f_m)$;

        (3) The unit element is $a_{1,1}\in \mathscr{P}_{set}(a_{1,1})={p'}_1^{-1}(a_{1,1})$;

        (4) The composition map 
        $$\gamma:\mathscr{P}_{set}(f)\times\mathscr{P}_{set}(g_1)\times\cdots\times \mathscr{P}_{set}(g_k)
        \to \mathscr{P}_{set}(f(g_1,\cdots,g_k)) $$
        is induced by the composition of elements in $\coprod A'_n$.
    \end{defn}

    Moreover, we let $\mathscr{P}$ be the ring operad in $(\mathbf{Cat},\times,*)$
    defined by $E\mathscr{P}_{set}$.

    With the same proof as Theorem \ref{thmsymbi} and Proposition
    \ref{propsymbiEinfty}, we get the following result.

    \begin{thm}\label{thmbiperm}
        The algebras over $\mathscr{P}$ are precisely bipermutative
        categories. Moreover, after applying the classifying space 
        functor, $B\mathscr{P}$ is an $E_\infty$ ring operad.
    \end{thm}
    \begin{proof}
        Note that $A'_n$ is a quotient of $A_n$ as a set. We define a 
        section map $s: A'_n\to A_n$ as follows.

        For each $a\in A_n$, we express $a$ as 
        a sequence with length $l$ of variables in $\{a_{1,n},\cdots,a_{n,n},1_n\}$,
        a sequence with length $l-1$ of operators $\{+,\times\}$,
        together with a parenthesization.
        For each operator $\bullet\in\{+,\times\}$ 
        in this sequence, $a$ must be locally of the form 
        \begin{align*}
            \cdots ((\alpha) \bullet (\beta))\cdots
        \end{align*}
        for some $\alpha,\beta\in A_n$.
        In this case, we say this operator $\bullet$
        acts on $(\alpha,\beta)$ in $a$.
        Moreover, we say $a\in A_n$ is reduced if 

        (1) for any operator $\times$ acting on 
        $\alpha, \beta$ in $a$, 
        \begin{align*}
            \alpha&\in\{a_{1,n},\cdots, a_{1,n}\},\\
            \beta&\not\in \{1_n,0_n\},
        \end{align*}
        
        (2) for any operator $+$ acting on 
            $\alpha', \beta'$ in $a$, 
            \begin{align*}
                \alpha'&\neq 0_n,\\
                \beta'&\neq 0_n.
            \end{align*}

        By applying the associativity and distributivity law,
        for each $b\in A'_n$, there exists a unique reduced $a\in A'_n$
        such that $\nu(a)=b$. We define $s: A'_n \to A_n$
        such that $s(x)$ is the unique reduced element in $\nu^{-1}(x)$.

        Regarding $\mathscr{P}(f)$ as the full sub-category of 
        reduced objects in
        $\mathscr{S}(f)$, we can modify the proof of Theorem
        \ref{thmsymbi} to show that algebras over $\mathscr{P}$ are precisely bipermutative
        categories.

        Moreover, since the $\widehat{\mathcal{R}}$ action 
        on $\mathscr{P}$ only changes the sequence of variables
        in the reduced expression of objects in $\mathscr{P}$,
        by the same proof of Proposition \ref{propsymbiEinfty},
        it follows that $B\mathscr{P}$ is an $E_\infty$ ring operad.
    \end{proof}

    \begin{cor}
        Let $C$ be a bipermutative category and let $0$ 
        be the based-point of the classifying space $BC$.
        Then there is a group completion
        \begin{align*}
            BC\to \Omega^\infty \mathbb{E}(BC)
        \end{align*}
        where $\mathbb{E}(BC)$ is an $E_\infty$ ring spectrum
        depending functorially on $C$.
    \end{cor}

    \begin{rem}\label{comptwoconst}
        So far we have two different constructions
        from bipermutative category to classical $E_\infty$
        ring space: one shown in \cite{TheconstructionofE_infringspacesfrombipermutativecategories}
        and the other is given by ring operad theory.
        Now we will describe a comparison between them.

        To begin with, note that Definition \ref{defnconsofringandcat}
        assigns to each $E_\infty$ ring operad a category of 
        ring operators, but this construction also work on the 
        category level. Starting with the ring operad 
        $\mathscr{P}$, applying the construction in 
        Definition \ref{defnconsofringandcat}, we 
        get a categorically enriched category of ring operators
        $\tilde{\mathscr{P}}$, that is, 
        a category enriched in $\mathbf{Cat}$, i.e. a 
        strict 2-category, 
        which is over 
        $\mathscr{F}\wr\mathscr{F}$ and under 
        $\Pi\wr\Pi$ satisfying the same conditions in Definition 
        \ref{defnofcatofringop}. 
        Moreover, we regard a $\tilde{\mathscr{P}}$-category as 
        a 2-functor from $\tilde{\mathscr{P}}$ to the 2-category
        of small categories and we define the notion 
        of a special $\tilde{\mathscr{P}}$-category similarly.
        Under the same argument as Definition
        \ref{defnringandcatonalg}, we also get a canonical functor $
        \nu:\mathscr{P}[\mathbf{Cat}_e]\to \tilde{\mathscr{P}}^s[\mathbf{Cat}]$
        from the category of $\mathscr{P}$-algebras in $\mathbf{Cat}$
        to the category special $\tilde{\mathscr{P}}$-categories.

        Now we summarize the comparison between the two 
        constructions by the following figure. 
        \begin{center}
            \begin{codi}
                \obj{
                |(11)|h(\text{bipermutative cats}) &[10.5em] |(12)| h(\mathscr{P}[\mathbf{Cat}_e]) &[9em] |(13)|h(\mathscr{P}[\mathbf{Cat}_e])\\
                |(21)|h(\text{special }\mathscr{F}\wr\mathscr{F}\text{-cats}) &|(22)|h(\text{special } \tilde{\mathscr{P}}\text{-cats}) &\\
                |(31)|h(\text{special }\mathscr{F}\wr\mathscr{F}\text{-spaces}) & |(32)|h(\text{special }\widetilde{B\mathscr{P}}\text{-spaces}) &|(33)|h(B\tilde{\mathscr{P}}\text{-spaces})\\
                |(41)|h(\text{special }\hat{\mathscr{L}}\wr\hat{\mathscr{K}}\text{-spaces})& |(42)|h(\text{special }\hat{\mathscr{L}}\wr\hat{\mathscr{K}}\times_{\mathscr{F}\wr\mathscr{F}}\widetilde{B\mathscr{P}}\text{-spaces} )&|(43)|h(\mathscr{R}_{\mathscr{K},\mathscr{L}}\text{-spaces})\\
                &|(52)|h((\mathscr{K},\mathscr{L})\text{-spaces})&\\
                &|(62)|h(E_\infty \text{ ring spectra})&\\
                };
                \mor :[shove=+.1em] 12 [=] 13;
                \mor :[shove=-.1em] 12 [=] 13;
                \mor 11 -> 21 "B":-> 31 "\simeq":-> 41;
                \mor 12 "\nu":-> 22 "B":-> 32 "\simeq":-> 42;
                \mor 13 "B":-> 33 -> 43;
                \mor 11 "\cong":-> 12;
                \mor 21 -> 22;
                \mor 31 "\simeq":-> 32;
                \mor 42 "\simeq":-> 41;
                \mor 33 "\simeq":-> 32;
                \mor 42 "\simeq":-> 43;
                \mor 41 "\simeq":-> 52;
                \mor 43 "\simeq":-> 52;
                \mor 52 -> 62;
            \end{codi}
        \end{center}

        Here $hC$ means the homotopy category of $C$.
        In particular, the category of special $\tilde{\mathscr{P}}$-categories
        is by definition a full sub-2-category of the 2-category
        of 2-functors between the 2-category $\tilde{\mathscr{P}}$
        and the 2-category of categories, but here we only consider its
        homotopy category.

        The whole diagram commutes automatically up to natural isomorphisms by definition
        except for the top left square.
        Moreover, each homotopy equivalence arrow in the above diagram
        induces an homotopy equivalence on the underlying spaces.
        Here the underlying space of a special $\mathscr{J}$-space for some
        category of ring operators $\mathscr{J}$ is its value on the object
        $(1,1)\in\mathscr{J}$.

        Note that the composite of all these left arrows 
        is precisely the construction in \cite{TheconstructionofE_infringspacesfrombipermutativecategories}
        and the composite of all these right arrows 
        is precisely the construction given by ring operad theory.
        Therefore, to compare these two constructions, we only need to 
        consider the top left square.
        \begin{prop}
            The following square commutes up to natural equivalence.
            As a corollary, the two constructions from 
            bipermutative categories to $E_\infty$ ring spectra
            coincide up to homotopy.
            \begin{center}
                \begin{codi}
                    \obj{
                    |(11)|h(\text{bipermutative cats}) &[10em] |(12)| h(\mathscr{P}[\mathbf{Cat}_e]) \\
                    |(21)|h(\text{special }\mathscr{F}\wr\mathscr{F}\text{-cats}) &|(22)|h(\text{special } \tilde{\mathscr{P}}\text{-cats})\\
                    };
                    \mor 11 -> 21;
                    \mor 12 "\nu":-> 22;
                    \mor 11 "\cong":-> 12;
                    \mor 21 -> 22;
                \end{codi}
            \end{center}
        \end{prop}
        \begin{proof}
            Recall that in \cite{may1982multiplicative},
            the passage from bipermutative categories
            to $\mathscr{F}\wr\mathscr{F}$-categories
            is constructed as follows. We first construct
            a lax functor from $\mathscr{F}\wr\mathscr{F}$ 
            to $\mathbf{Cat}$, and then apply
            \cite[Theorem 3.4]{MAY1980299} to strictify
            it to a genuine functor.

            Actually this proposition holds directly 
            from \cite[Theorem 3.4]{MAY1980299}. 
            To see this, note that each 
            2-functor $F$ from $\tilde{\mathscr{P}}$
            to the 2-category of categories determines a 
            lax functor $\bar{F}$ from $\mathscr{F}\wr\mathscr{F}$ 
            to $\mathbf{Cat}$ as follows:

            (1) $\bar{F}(n,S)=F(n,S)$;

            (2) 
            For any $f\in\coprod\mathcal{R}(n)$,
            say 
            $$f=\sum_{I=(i_1,i_2,\cdots,i_m)\in \{0,1\}^m\setminus\{0\}^m} 
            \varepsilon_I a_{1,m}^{i_1}\cdots a_{m,m}^{i_m}\in \coprod\mathcal{R}(n).$$
            Let 
            $t(f):=\sum_{I=(i_1,i_2,\cdots,i_m)\in \{0,1\}^m\setminus\{0\}^m} 
            \varepsilon_I a_{1,m}^{i_1}\cdots a_{m,m}^{i_m}\in\mathrm{Obj}(\mathscr{P}(f))$
            under the canonical order of both $\Lambda_f$ and $\Gamma_I$.
            By definition,
            for each morphism $(\phi,d):(m,R)\to (n,S)$ in 
            $\mathscr{F}\wr\mathscr{F}$, the preimage of $(\phi,d)$
            under the canonical map $\tilde{\mathscr{P}}\to\mathscr{F}\wr\mathscr{F}$
            is 
            $$\prod_{h,j}\mathscr{P}(f_{\phi,d,h,j}).$$
            We define $\bar{F}(\phi,d):=F((tf_{\phi,d,h,j})_{h,j})$.

            (3) $\sigma((\phi,d),(\psi,d')):\bar{F}(\phi,d)\circ\bar{F}(\phi',d')\to\bar{F}((\phi,d)\circ(\phi',d'))$
            is defined to be the canonical equivalence.

            Let $(A,\oplus,\otimes,0,1)$ be an arbitrary 
            bipermutative category. Then 
            the composite of the top and the right arrow
            in this diagram sends $A$ to a $\tilde{\mathscr{P}}$-category
            $A'$, which determines a lax functor $\bar{A'}$.
            By carefully checking definitions, we have this 
            lax functor $\bar{A'}$ coincides with that appears in the 
            passage from bipermutative categories
            to $\mathscr{F}\wr\mathscr{F}$-categories
            in \cite{may1982multiplicative}.
            Therefore, this proposition holds directly 
            from the strictification theorem \cite[Theorem 3.4]{MAY1980299}
            from lax functor to genuine functor since the strictification 
            of a lax functor doesn't change homotopy type.
        \end{proof}
    \end{rem}

    We end this section with a proof of the strictification
    theorem from symmetric bimonoidal categories
    to bipermutative categories. Recall that when $\nu': \mathscr{C}\to \mathscr{C}'$
    is a morphism between $E_\infty$ ring operads, 
    any $\mathscr{C}$-algebra is homotopy equivalent to some 
    $\mathscr{C}'$-algebra.
    A similar result holds for the functor 
    $\nu: \mathscr{S}\to \mathscr{P}$. Of course, we
    can prove this by modifying the proof in Section 
    \ref{Sectiondiffoperad}. However, in the special 
    case $\nu: \mathscr{S}\to \mathscr{P}$,
    a simpler proof is possible.

    \begin{thm}\label{thmstrictification}
        There is a functor $\Phi$ from the category of
        tight symmetric bimonoidal categories
        (with strict unit and zero objects) to the category of bipermutative categories 
        and a natural equivalence $\eta: \Phi C \to C$ of symmetric bimonoidal
        categories. 
    \end{thm}

    \begin{proof}
        Let $\mathbb{S}$, $\mathbb{P}$ be the monads in $\mathbf{Cat}$ associated
        to $\mathscr{S}$, $\mathscr{P}$, respectively.

        By definition of the section map $s:\mathscr{P}\to \mathscr{S}$
        sending each object $x$ to the unique reduced object in $\nu^{-1}(x)$,
        we have $s:\mathscr{P}\to \mathscr{S}$ is a natural transformation 
        between functors over $\widehat{\mathcal{R}}$, although
        it is not a morphism between ring operads.
        Therefore, there is a well-defined functor 
        $s:\mathbb{P} X \to \mathbb{S} X$.

        When $X$ is symmetric bimonoidal,
        we define 
        \begin{center}
            \begin{codi}
                \obj{
                    |(0)|\tilde{\eta}: \mathbb{P} X
                    &|(1)|\mathbb{S} X
                    &|(2)|X.\\
                    };
                \mor 0 "s":-> 1 "\theta":-> 2;
            \end{codi}
        \end{center}

        Moreover, we define a relation on each hom-set 
        $\mathbb{P} X(A,B)$ such that two morphisms
        $f,g\in \mathbb{P} X(A,B)$ are equivalent if and only if 
        $\tilde{\eta}(f)=\tilde{\eta}(f)$. Then we denote the 
        quotient category by $\Phi X$.

        Because $\tilde{\eta}: \mathbb{P} X \to X$ is symmetric bimonoidal,
        we get the following commutative diagram in the category of 
        symmetric bimonoidal categories:

        \begin{center}
            \begin{codi}
                \obj{
                    |(0)|\mathbb{P} X
                    &[5em] |(1)|X\\
                    |(2)|\Phi X
                    &\\
                    };
                \mor 0 "\tilde{\eta}":-> 1;
                \mor 0 "\text{quotient}":-> 2 "\eta":-> 1;
            \end{codi}
        \end{center}

        Here $\eta$ is faithful by construction and it is full and 
        essentially surjective since 
        \begin{center}
            \begin{codi}
                \obj{
                    |(0)|X
                    &[3em]|(1)|\mathscr{P}(a_{1,1})\times X
                    &[3em]|(2)|\mathbb{P} X
                    &[2em]|(3)|\Phi X 
                    &|(4)| X\\
                    };
                \mor 0 "\text{unit}":-> 1 -> 2 "\text{quotient}":-> 3 "\eta":-> 4;
            \end{codi}
        \end{center}
        is the identity. Therefore, the theorem holds.
    \end{proof}

    \begin{rem}
        Our construction is different from \cite[Section VI.3]{may1977ring},
        but there is a canonical comparison equivalence
        $$\Phi X \to \Phi'X$$
        where $\Phi'X$ is the bipermutative category constructed in \cite{may1977ring}. 
    \end{rem}

    \appendix
    \section{Proof of the comparison theorem}\label{appendix}
    
    \subsection{Properties of the category $\widehat{\mathcal{R}}_{eff}$}\label{Sectioncomb}

    In this section, we describe the structure of 
    $\widehat{\mathcal{R}}_{eff}$ more precisely. These
    combinatorial descriptions will be used in the 
    next section to prove the Comparison Theorem \ref{thmcompare1}.
    
    Recall that $(f_m,\phi,f_n)\in\widehat{\mathcal{R}}_{eff}(f_m,f_n)$
    if and only if 
    $$f_n(a_{1,n},\cdots,a_{n,n})=f_m(a_{\phi(1),n},\cdots,a_{\phi(m),n}).$$
    In general, the hom-set $\widehat{\mathcal{R}}_{eff}(f_m,f_n)$
    between two arbitrary objects $f_m$ and $f_n$ might
    be empty.
    We regard a morphism from $f_m$ to $f_n$ as a relation 
    between them and two related objects should have something in common.
    
    More precisely, we say two objects in a small category are connected
    if they are path-connected in the classifying space; that is,
    there is a zig-zag diagram of morphisms connecting these two 
    objects. Now, we describe the connected components
    of $\widehat{\mathcal{R}}_{eff}$ precisely.

    To begin with, by Lemma \ref{type}, if two objects
    $f_m,f_n$ are contained in the same connected component, we must have:

    (1) $|\Lambda_{f_m}|=|\Lambda_{f_n}|$;

    (2) the ordered $|\Lambda_{f_m}|$-tuple 
    $(|\Gamma_I|)_{I\in\Lambda_{f_m}}$ is a permutation of 
    $(|\Gamma_J|)_{J\in\Lambda_{f_n}}$.

    For a non-negative integer $k$ and a non-decreasing sequence
    of $l$ positive integers 
    $0<k_1\leqslant i_2\leqslant \cdots\leqslant k_l$, 
    we say an object $f$ in $\widehat{\mathcal{R}}_{eff}$ is of 
    type $(l;k_1,\cdots,k_l)$ if and only if 
    $|\Lambda_{f}|=l$ and $(|\Gamma_I|)_{I\in\Lambda_{f_m}}$
    is a permutation of $(k_1,\cdots,k_l)$. Therefore,
    two connected objects are of the same type.

    In fact, the converse statement is also true.
    To see this, note that by definition \ref{special}, 
    the special object 
    $$f=a_{1,n}\cdots a_{k_1,n}+a_{k_1+1,n}\cdots a_{k_1+k_2,n}+\cdots+a_{k_1+\cdots+k_{l-1}+1,n}\cdots a_{k_1+\cdots+k_{l},n}$$
    is of type $(l;k_1,\cdots,k_l)$ and two different
    special objects are of different types. More precisely, we have the 
    following result.

    \begin{prop} \label{propofspecial}
        For each $f\in\coprod \mathcal{R}(n)$,
    there exists a unique special $g\in\coprod \mathcal{R}(n)$ such that 
    $\widehat{\mathcal{R}}_{eff}(g,f)$ is non-empty. Moreover, 
    two different special objects are in different connected
    components. Therefore, the collection of special objects
    gives a collection of representatives for connected
    components in $\widehat{\mathcal{R}}_{eff}$.
    \end{prop}
    \begin{proof}
    Different special objects are in different connected
    components because they are of different types, so it suffices
    to construct for any object $f$ a morphism from the 
    special object of the same type as $f$ to $f$ itself.

    Indeed, suppose
    $$f=\sum_{I=(i_1,i_2,\cdots,i_n)\in \{0,1\}^n} 
        \varepsilon_I a_{1,n}^{i_1}\cdots a_{n,n}^{i_n}$$
    is of type $(l,k_1,\cdots,k_l)$. Then we want to 
    find a morphism from the special object
    $$g=a_{1,n}\cdots a_{k_1,n}+a_{k_1+1,n}\cdots a_{k_1+k_2,n}+\cdots+a_{k_1+\cdots+k_{l-1}+1,n}\cdots a_{k_1+\cdots+k_{l},n}$$
    to $f$.

    Let $$\Lambda_f=:\{I_1,\cdots,I_l\}$$ be such that 
    $|\Gamma_{I_j}|=k_j$ and let $$I_j=\{t_{1,j}<t_{2,j}<\cdots<t_{k_j,j}\}.$$
    Then 
    \begin{align*}
        \phi: \{1,2,\cdots,k_1+\cdots+k_l\}&\to \{1,2,\cdots,n\}\\
        \sum_{h=1}^{j-1}k_h+s&\mapsto t_{s,j}\text{ if } 1\leqslant s \leqslant k_j
    \end{align*}
    gives a well-defined morphism $(g,\phi,f)\in\widehat{\mathcal{R}}(g,f)$
    by definition.
    \end{proof}

    Now we focus on these non-degenerate objects. Actually, those 
    degenerate objects appear in the definition of ring operad
    only because the non-degenerate ones are not closed under composition.
    In practice, for a ring operad $\mathscr{C}$,
    the collection of $\mathscr{C}(f)$ indexed by 
    non-degenerate objects $f$ already provides enough information.
    For example, $\mathscr{C}(a_{1,3}a_{2,3}+a_{1,3}a_{3,3})$
    is regarded as a collection of operators of the form 
    $$(a,b,c)\to (ab+ac),$$
    and $\mathscr{C}(a_{1,4}a_{2,4}+a_{1,4}a_{3,4})$
    is regarded as a collection of operators of the form 
    $$(a,b,c,d)\to (ab+ac).$$
    Therefore, we do not expect $\mathscr{C}(a_{1,4}a_{2,4}+a_{1,4}a_{3,4})$
    to contain more information than $\mathscr{C}(a_{1,3}a_{2,3}+a_{1,3}a_{3,3})$.
    This interpretation leads to one of the conditions in Definition \ref{defnEinfty}
    of $E_\infty$ ring operads.

    Now we describe the structure of the full sub-category $\widehat{\mathcal{R}}_{n.d.}$ of 
    $\widehat{\mathcal{R}}_{eff}$ generated by non-degenerate objects.

    Consider morphisms in $\widehat{\mathcal{R}}_{n.d.}$ first. We 
    have the following lemma.
    \begin{lem}\label{lemofnondeg}
        Let $(f,\phi,g)$ be an effective morphism in $\widehat{\mathcal{R}}_{eff}$
        with $f$ non-degenerate.
        Then $g$ is non-degenerate if and only if 
        $\phi:\{0,e,1,2,\cdots,|f|\}\to \{0,e,1,2,\cdots,|g|\}$ is surjective.
        In this case, $|f|\geqslant|g|$. 
        
        In particular, for any map 
        $\psi:\{0,e,1,2,\cdots,|f|\}\to \{0,e,1,2,\cdots,|f|\}$,
        $(f,\psi,\psi_*f)$ defines a morphism in $\widehat{\mathcal{R}}_{n.d.}$
        if and only if $\psi$ is a bijection.
    \end{lem}
    \begin{proof}
        Suppose $\phi$ is not surjective.
        Pick a $k\in \{1,2,\cdots,|g|\}\setminus \mathrm{Im}(\phi)$. Then 
        since 
        $$g(a_{1,|g|},\cdots,a_{|g|,|g|})=f(a_{\phi(1),|g|},\cdots,a_{\phi(|g|),|g|}),$$ 
        $\frac{\partial}{\partial a_{k,|g|}}g=0$, hence
        $g$ is degenerate.

        Conversely, if $\phi$ is surjective and $\psi_*f=g\in\mathcal{R}(|g|)$, then
        $\frac{\partial}{\partial a_{k,|f|}}\psi_*f\neq 0$
        follows from $\frac{\partial}{\partial a_{\phi^{-1}(k),|f|}}f\neq 0$,
        so $\psi_*f$ is non-degenerate.
        
        In particular, note that a bijection $\psi$ on $\{1,2,\cdots,|f|\}$
        induces a bijection on monomials in $\mathcal{R}(|f|)$,
        so we get $\psi_*f\in\mathcal{R}(|f|)$. Therefore, the lemma holds.
    \end{proof}

    We end this section with a description of the connected components
    of $\widehat{\mathcal{R}}_{n.d.}$.
    As shown in Proposition \ref{propofspecial}, each non-degenerate
    object $f$ admits a morphism from the special object $g$
    of the same type as $f$ to $f$ itself. Such a morphism exists implies
    $|g|\geqslant|f|$. Therefore, the following results hold.

    \begin{notn}
        Let $f$ be a special object. Let $\widehat{\mathcal{R}}_{n.d.}(f)$
    be the connected component containing $f$ with objects $\mathcal{R}_{n.d.}(f)$ and 
    $\widehat{\mathcal{R}}_{n.d.}(\geqslant n)$ ($\widehat{\mathcal{R}}_{n.d.}(n)$, resp.)
    be the sub-category generated by non-degenerate $g$ with $|g|\geqslant n$
    ($|f|=n$, resp.)
    with objects $\mathcal{R}_{n.d.}(\geqslant n)$ ($\mathcal{R}_{n.d.}(n)$, resp.). Denote 
    \begin{align*}
        \widehat{\mathcal{R}}_{n.d.}(f,\geqslant n)&:=\widehat{\mathcal{R}}_{n.d.}(f)\cap \widehat{\mathcal{R}}_{n.d.}(\geqslant n),\\
        \widehat{\mathcal{R}}_{n.d.}(f,n)&:=\widehat{\mathcal{R}}_{n.d.}(f)\cap \widehat{\mathcal{R}}_{n.d.}(n).
    \end{align*}
    with objects
    \begin{align*}
        \mathcal{R}_{n.d.}(f,\geqslant n)&:=\mathcal{R}_{n.d.}(f)\cap \mathcal{R}_{n.d.}(\geqslant n),\\
        \mathcal{R}_{n.d.}(f,n)&:=\mathcal{R}_{n.d.}(f)\cap \mathcal{R}_{n.d.}(n).
    \end{align*}
    \end{notn}

    \begin{prop}\label{propofnondegen}
        The collection of special objects
    gives a collection of representatives for connected
    components in $\widehat{\mathcal{R}}_{n.d.}$.

    Moreover, there is a finite filtration on each connected component.
    
    Then we get 
    $$\emptyset=\widehat{\mathcal{R}}_{n.d.}(f,\geqslant |f|+1)\subset \widehat{\mathcal{R}}_{n.d.}(f,\geqslant |f|)\subset\cdots
    \subset \widehat{\mathcal{R}}_{n.d.}(f,\geqslant 0)=\widehat{\mathcal{R}}_{n.d.}(f).$$
    \end{prop}

    This filtration plays an essential role in section \ref{Sectiondiffoperad}.

    \subsection{Algebras over different $E_\infty$ ring operads}\label{Sectiondiffoperad}

    We finally prove the Comparison Theorem \ref{thmcompare1}
    in this section. To show the proof, we need to introduce 
    two lemmas first. 
    \begin{lem}\label{lemglue}
        Given the following commutative diagram of spaces
        \begin{center}
            \begin{codi}
                \obj {A &[-1em]   &   &[-1em] B\\[-2em]
                        & A'&B' &\\[-1em]
                        & C'&D' &\\[-2em]
                    C&&&D\\
                };
                \mor A -> B -> D;
                \mor A "i":-> C -> D;
                \mor A' -> B' -> D';
                \mor A' "i'":-> C' -> D';
                \mor A "\alpha":-> A';
                \mor B "\beta":-> B';
                \mor C "\gamma":-> C';
                \mor D "\delta":-> D';
            \end{codi}
        \end{center}
        in which $i$ and $i'$ are cofibrations, $\alpha$, $\beta$, $\gamma$ are 
        equivalences, both the outer and inner squares are 
        push-out squares, then $\delta$ is also an equivalence.
    \end{lem}

    \begin{proof}
        We functorially factorize all horizontal morphisms in the 
        above diagram as a composition of an 
        acyclic fibration and a cofibration. Therefore it suffices to prove
        this lemma when all horizontal morphisms are also cofibrations.
        This follows from \cite[Pages 80-81]{may1999concise}.
    \end{proof}

    The next lemma is also a fundamental result 
    in the theory of model category, see \cite{Lellig1973}.

    \begin{lem}\label{lemofunion}
        Given the following pull-back diagram of spaces
        \begin{center}
            \begin{codi}
                \obj {A & B\\[-2em]
                    C & X\\
                };
                \mor A -> B -> X;
                \mor A -> C -> X;
            \end{codi}
        \end{center}
        in which morphisms $B\to X$, $C\to X$, and the composition
        $A\to X$ are cofibrations, then so is the 
        universal morphism $$D:=B\coprod_{A} C\to X.$$
    \end{lem}

    Applying these lemmas, we get the following property about $E_\infty$ ring operads.

    \begin{prop}
        Let $\mathscr{C}$ be an $E_\infty$ ring operad, $f$ be a special object, 
        and $g\in\mathcal{R}(f)$. 
        Let $\{\phi_i(\mathscr{C}(h_i))\}_{i=1}^N$ be a finite collection
        of distinct sub-spaces in $\mathscr{C}(g)$ where $(h_i,\phi_i,g)$
        is a morphism in $\widehat{\mathcal{R}}_{n.d.}$.
        Then both $\bigcap_{i=1}^N\phi_i(\mathscr{C}(h_i))$
        and $\bigcup_{i=1}^N\phi_i(\mathscr{C}(h_i))$ are contractible
        and their inclusions into $\mathscr{C}(g)$ are cofibrations.
    \end{prop}

    \begin{proof}
        We prove this by induction first on $|g|$ and then on $\min\{|h_i|\}$.

        The induction begins with $\min\{|h_i|\}=|f|$.
        Note that for $h_i\in\mathcal{R}(f)_{n.d.}$ with 
        $|h_i|=|f|$, there must be some $\sigma_i\in\Sigma_{|f|}$
        such that $h_i=\sigma_{i*}f$. 
        $\sigma_{i*}:\mathscr{C}(f)\to \mathscr{C}(h_i)$ are homeomorphisms,
        so without lost of generality, we assume all $h_i=f$.

        By Lemma \ref{lemofspecialpullback}, two morphisms 
        $(f,\phi_i,g)$ and $(f,\phi_j,g)$ only differ by an 
        automorphism of $f$, which induces a homeomorphism on $\mathscr{C}(f)$.
        So $\phi_i(\mathscr{C}(h_i))$ are all the same space in this case. 
        Therefore, the proposition holds by condition (5) in Definition 
        \ref{defnEinfty} when $\min\{|h_i|\}=|f|$, and hence when $|g|=|f|$.

        Now we assume the proposition holds when $|g'|>|g|$ and when 
        $|g'|=|g|$ with $\min\{|h'_i|\}>\min\{|h'_i|\}$.
        There is nothing to prove when $N=1$ since it follows from 
        condition (5).

        When $N\geqslant 2$, for any $\alpha\in \phi_i(\mathscr{C}(h_i))\cap\phi_j(\mathscr{C}(h_j))$
        for $i\neq j$, say $\alpha=\phi_{i*}\alpha_1=\phi_{j*}\alpha_j$.
        By condition (3) in Definition \ref{defnEinfty}, 
        there exists $h'_{ij}$ with morphisms $(h'_{ij},\psi_i,h_i)$
        $(h'_{ij},\psi_j,h_j)$ and $\beta\in\mathscr{C}(h'_{ij})$ 
        such that $\psi_{i*}\beta=\alpha_i$, $\psi_{j*}\beta=\alpha_j$.
        By condition (4), $\phi_i\psi_{i}=\phi_j\psi_{j}=:\phi_{ij}$.
        By Lemma \ref{lemofnondeg}, $|h'_{ij}|\geqslant \max\{|h_i|,|h_j|\}$.
        If $|h'_{ij}|=|h_i|=|h_j|$, then both $\psi_{i},\psi_{j}$ are invertible,
        so 
        $$\phi_i(\mathscr{C}(h_i))=\phi_{ij}(\mathscr{C}(h'_{ij}))=\phi_j(\mathscr{C}(h_j)),$$
        which contradicts with $i\neq j$.

        Therefore, $|h'_{ij}|> \min\{|h_i|,|h_j|\}$, so 
        $$\bigcap_{i=1}^N\phi_i(\mathscr{C}(h_i))=\bigcap_{i,j}\phi_{ij}(\mathscr{C}(h'_{ij}))$$
        with $\min\{|h'_{ij}|\}> \min\{|h_{i}|\}$. Hence,
        $\bigcap_{i=1}^N\phi_i(\mathscr{C}(h_i))$ 
        is contractible
        and its inclusion into $\mathscr{C}(g)$ is a cofibration.

        To prove the union part, we further assume that 
        $|g|>\min\{|h_i|\}$. Otherwise, if $|h_i|=|g|$ for some $i$
        then $\phi_i$ is invertible, so $\phi_{i*}:\mathscr{C}(h_i)\to \mathscr{C}(g)$
        is a homeomorphism. Thus the union 
        $\bigcup_{i=1}^N\phi_i(\mathscr{C}(h_i))$ is the whole space 
        $\mathscr{C}(g)$ and the proposition holds.

        Consider the following commutative diagram
        \begin{center}
            \begin{codi}
                \obj {|(1)|\phi_k(\mathscr{C}(h_k))\cap\biggl(\bigcup_{i=1}^{k-1}\phi_i(\mathscr{C}(h_i))\biggr) 
                &[10em] |(2)|\bigcup_{i=1}^{k-1}\phi_i(\mathscr{C}(h_i))\\
                |(3)|\phi_k(\mathscr{C}(h_k)) 
                & |(4)|\bigcup_{i=1}^{k}\phi_i(\mathscr{C}(h_i))\\
                };
                \mor 1 -> 2 -> 4;
                \mor 1 -> 3 -> 4;
            \end{codi}
        \end{center}
        This diagram is both a pull-back and a push-out diagram.
        Here 
        \begin{align*}
            &\phi_k(\mathscr{C}(h_k))\cap\biggl(\bigcup_{i=1}^{k-1}\phi_i(\mathscr{C}(h_i))\biggr)\\
            =&\bigcup_{i=1}^{k-1}(\phi_k(\mathscr{C}(h_k))\cap\phi_i(\mathscr{C}(h_i)))
        \end{align*}
        and $\phi_k(\mathscr{C}(h_k))\cap\phi_i(\mathscr{C}(h_i))$
        is a union of some $\bigcap_{i,j}\phi_{ij}(\mathscr{C}(h'_{ij}))$
        with $\min\{|h'_{ij}|\}>|g|$. Then the proposition holds for 
        $\phi_k(\mathscr{C}(h_k))\cap\biggl(\bigcup_{i=1}^{k-1}\phi_i(\mathscr{C}(h_i))\biggr)$.

        The left arrow in the above diagram is a cofibration
        by applying the inductive hypothesis to $g=h_k$ (note that $|h_k|>|g|$).
        Therefore, the proposition follows by applying
        Lemma \ref{lemglue} and \ref{lemofunion} inductively on the above diagram.
    \end{proof}

    In particular, we have the following corollary.
    \begin{cor}
        Let $\mathscr{C}$ be an $E_\infty$ ring operad.
        For a special object $f$ and $n\geqslant0$, let $L(f,n)$
        be the subspace 
        $$L(f,n)\subset \coprod_{g\in\mathcal{R}_{n.d.}(f,n)}\mathscr{C}(g),$$
        such that $\alpha\in \mathscr{C}(g)$ belongs to $L(f,n)$
        if and only if there exists 
        $$h\in \mathcal{R}_{n.d.}(f,\geqslant n+1)$$
        with morphism $(h,\phi,g)$ in $\widehat{\mathcal{R}}_{n.d.}$
        and $\beta\in\mathscr{C}(h)$
        such that $\phi_*\beta=\alpha$.
        Then each $L(f,n)\cap \mathscr{C}(g)$ is contractible
        for any $g\in\mathcal{R}_{n.d.}(f)$ and 
        the inclusion $L(f,n)\subset \coprod_{g\in\mathcal{R}_{n.d.}(f,n)}\mathscr{C}(g)$
        is a cofibration.
    \end{cor}

    Now let $\mathscr{C}$ be an $E_\infty$ ring operad and we focus on 
    the associated monad $\mathbb{C}$. Our argument here is closely related
    to \cite[Appendix]{May_1974}, but is much more complicated. The following
    lemma gives a strict description of the
    interpretation that only those spaces indexed by non-degenerate
    objects provide essential information.
    \begin{lem}\label{lemofndhomeo}
        If $\mathscr{C}$ is an $E_\infty$ ring operad, then the 
        canonical map
        \begin{align*}
            i:\mathscr{C}(\bullet)\times_{\widehat{\mathcal{R}}_{n.d.}^{op}}X^\bullet \to \mathscr{C}(\bullet)\times_{\widehat{\mathcal{R}}_{eff}^{op}}X^\bullet=:\mathbb{C}_{\mathscr{U}}X
        \end{align*}
        is a homeomorphism.
    \end{lem}
    \begin{proof}
        By definition of coend, 
        \begin{align*}
            \mathscr{C}(\bullet)\times_{\widehat{\mathcal{R}}_{eff}^{op}}X^\bullet=
            \coprod_{f\in\mathcal{R}(n),\ n\geqslant 0} \mathscr{C}(f)\times X^n/(\sim )
        \end{align*}
        where $(\sim)$ is generated by 
        $(\alpha,\phi^*\mathbf{x})\sim (\phi_*\alpha,\mathbf{x})$
        for all effective morphisms $(f,\phi,g)$ with $\alpha\in\mathscr{C}(f)$
        and $\mathbf{x}\in X^{|g|}$.

        We define an inverse $j$ of $i$ as follows.
        For any $(\alpha,\mathbf{x})\in \mathscr{C}(f)\times X^{|f|}$
        if $f$ is non-degenerate, we define $j(\alpha,\mathbf{x})$
        to be the quotient of $(\alpha,\mathbf{x})$ in 
        $\mathscr{C}(\bullet)\times_{\widehat{\mathcal{R}}_{eff}^{op}}X^\bullet$.

        When $f$ is degenerate, then we pick an effective morphism
        $(g,\phi,f)$ such that $g$ is special by Proposition \ref{propofspecial}.
        We further decomposite $\phi=\psi\circ p$ such that $\psi$ is injective
        and $p$ is surjective. Then by Lemma \ref{lemofinj}, 
        $\psi_*p_*g\in\mathrm{Obj}(\widehat{\mathcal{R}})$ and $\psi$ injective implies
        $p_*g\in\mathrm{Obj}(\widehat{\mathcal{R}})$. 
        Furthermore, $p$ surjective and $g$ non-degenerate implies $p_*g$ non-degenerate.
        Therefore, for any degenerate $f$, there exists a non-degenerate
        object $p_*g$ and a morphism $(p_*g,\psi,f)$ in which $\psi$ is injective.

        Note that the pair $(p_*g, \psi)$ is not necessarily unique, so we just 
        choose such a pair for each degenerate $f$ and denote it by 
        $(h_f,\psi_f)$. By assumption, $\mathscr{C}$ is $E_\infty$, so 
        $\psi_{f*}:\mathscr{C}(h_f)\to \mathscr{C}(f)$ is a homeomorphism.
        Then we define $j(\alpha,\mathbf{x}):=((\psi_{f*})^{-1}\alpha,\psi_{f}^*\mathbf{x})$.

        To check $j:\mathscr{C}(\bullet)\times_{\widehat{\mathcal{R}}_{eff}^{op}}X^\bullet\to 
        \mathscr{C}(\bullet)\times_{\widehat{\mathcal{R}}_{n.d.}^{op}}X^\bullet$
        is well-defined, it suffices to check 
        $j(\alpha,\phi^*\mathbf{x})= j(\phi_*\alpha,\mathbf{x})$
        for all effective morphisms $(f,\phi,g)$ with $\alpha\in\mathscr{C}(f)$
        and $\mathbf{x}\in X^{|g|}$.

        This holds by definition when both $f$ and $g$ are non-degenerate.
        If $f$ is non-degenerate but $g$ is degenerate, we get 
        $\phi_*\alpha=\psi_{g*}(\psi_{g*})^{-1}\phi_*\alpha$. 
        
        By definition, $\mathrm{Im}(\phi)$ is contained in $\mathrm{Im}(\phi_g)$,
        Moreover, $(f,\phi,g)$ factors through 
        $(f,\phi,g)=(h_g,\phi_g,g)\circ(f,\tilde{\phi},h_g)$.

        Therefore, 
        \begin{align*}
            j(\alpha,\phi^*\mathbf{x})=&j(\alpha,(\phi_g\tilde{\phi})^*\mathbf{x})\\
            =&j(\alpha,\tilde{\phi}^*\phi_g^*\mathbf{x})\\
            =&j(\tilde{\phi}_*\alpha,\phi_g^*\mathbf{x})\\
            =&j((\phi_g^{-1})_*\phi_{g*}\tilde{\phi}_*\alpha,\phi_g^*\mathbf{x})\\
            =&j((\psi_{g*})^{-1}\phi_*\alpha,\psi_{g}^*\mathbf{x})\\
            =&j(\phi_*\alpha,\mathbf{x}).
        \end{align*}

        The case when $f$ is degenerate can be reduced to the above cases
        by composing with $\phi_f$.

        Therefore, $j$ is well-defined. By definition of $j$, it follows that 
        both $i\circ j$ and $j\circ i$ are identities, so $i$ is a homeomorphism.        
    \end{proof}

    In the rest of this section, we define 
    $\mathbb{C}_{\mathscr{U}}X$ to be 
    $\mathscr{C}(\bullet)\times_{\widehat{\mathcal{R}}_{n.d.}^{op}}X^\bullet$.

    We apply the filtration defined in Proposition \ref{propofnondegen} to get 
    a filtration on $\mathbb{C}_{\mathscr{U}}X$.

    \begin{lem}\label{lemonCX}
        Let $\mathscr{C}$ be an $E_\infty$ ring operad with associated monad $\mathbb{C}_{\mathscr{U}}$.
        Then 
        $$\mathbb{C}_{\mathscr{U}}X=\coprod_{f \text{ special}}\mathscr{C}(\bullet)\times_{\widehat{\mathcal{R}}_{n.d.}(f)^{op}}X^\bullet$$
        Moreover, for each $f$ special and 
        $n=0,1,\cdots,|f|-1$ there is a push-out diagram 
        \begin{center}
            \begin{codi}
                \obj {|(0)|L(f,n)\times_{\Sigma_n^{op}}X^n
                &[12em]|(1')| \mathscr{C}(\bullet)\times_{\widehat{\mathcal{R}}_{n.d.}(f,\geqslant n+1)^{op}}X^\bullet\\
                |(1)| \biggl(\displaystyle\coprod_{g\in\mathcal{R}_{n.d.}(f,n)}\mathscr{C}(g)\biggr)\times_{\Sigma_n^{op}}X^n&
                |(2)| \mathscr{C}(\bullet)\times_{\widehat{\mathcal{R}}_{n.d.}(f,\geqslant n)^{op}}X^\bullet\\
                };
                \mor 0 "e":-> 1';
                \mor 0 -> 1;
                \mor 1 -> 2;
                \mor 1' -> 2;
            \end{codi}
        \end{center}
    \end{lem}
    \begin{proof}
        The attaching map $e$ is defined as follows.
        \begin{align*}
            e: L(f,n)\times_{\Sigma_n^{op}}X^n&\to \mathscr{C}(\bullet)\times_{\widehat{\mathcal{R}}_{n.d.}(f,\geqslant n+1)^{op}}X^\bullet\\
            (\phi_*\beta, \mathbf{X})&\mapsto (\beta,\phi^{*}\mathbf{X})
        \end{align*}
        for some morphism $(h,\phi,g)$ in $\widehat{\mathcal{R}}_{n.d.}$
        and $\beta\in\mathscr{C}(h)$ with $h\in\mathcal{R}_{n.d.}(f,\geqslant n+1)$,
        $g\in\mathcal{R}_{n.d.}(f,n)$.

        This attaching map $e$ is well-defined. Indeed,
        for any $\sigma\in\Sigma_n$, we have 
        $$e(\sigma_*\phi_*\beta, \mathbf{X})=(\beta, \phi^*\sigma^*\mathbf{X})=e(\phi_*\beta, \sigma^*\mathbf{X}).$$
        Also, if there exist another 
        $(h',\phi',g)$ in $\widehat{\mathcal{R}}_{n.d.}$
        and $\beta'\in\mathscr{C}(h')$ with $h\in\mathcal{R}_{n.d.}(f,\geqslant n+1)$
        such that $ \phi'_*\beta'=\phi_*\beta$, then 
        there exists some non-degenerate $h''$ and $\beta''\in\mathscr{C}(h'')$ together with 
        morphisms $(h'',\psi,h)$, $(h'',\psi',h')$ such that $\psi_*\beta''=\beta$,
        $\psi'_*\beta''=\beta'$. By Lemma \ref{lemofnondeg}, $|h''|\geqslant |h|$,
        so $h''\in \mathcal{R}_{n.d.}(f,\geqslant n+1)$.
        By condition (4) in Definition \ref{defnEinfty}, 
        $\phi_*\psi_*\beta''=\phi'_*\psi'_*\beta''$ implies 
        $\phi\psi=\phi'\psi'$, so the following equation holds in 
        $\mathscr{C}(\bullet)\times_{\widehat{\mathcal{R}}_{n.d.}(f,\geqslant n+1)^{op}}X^\bullet$.
        \begin{align*}
            (\beta,\phi^{*}\mathbf{X})=&(\psi_*\beta'',\phi^{*}\mathbf{X})\\
            =&(\beta'',\psi^*\phi^{*}\mathbf{X})\\
            =&(\beta'',\psi'^*\phi'^{*}\mathbf{X})\\
            =&(\psi'_*\beta'',\phi'^{*}\mathbf{X})\\
            =&(\beta',\phi'^{*}\mathbf{X}).
        \end{align*}

        Note that 
        $$\biggl(\displaystyle\coprod_{g\in\mathcal{R}_{n.d.}(f,n)}\mathscr{C}(g)\biggr)\times_{\Sigma_n^{op}}X^n=
        \mathscr{C}(\bullet)\times_{\widehat{\mathcal{R}}_{n.d.}(f,n)^{op}}X^\bullet.$$
        So 
        \begin{align*}
            &\mathscr{C}(\bullet)\times_{\widehat{\mathcal{R}}_{n.d.}(f,\geqslant n)^{op}}X^\bullet\\
            =&
            \biggl(\displaystyle\coprod_{g\in\mathcal{R}_{n.d.}(f,n)}\mathscr{C}(g)\biggr)\times_{\Sigma_n^{op}}X^n
            \coprod\mathscr{C}(\bullet)\times_{\widehat{\mathcal{R}}_{n.d.}(f,\geqslant n+1)^{op}}X^\bullet/(\sim )
        \end{align*}
        where $(\sim)$ is generated by 
        $(\alpha,\phi^*\mathbf{x})\sim (\phi_*\alpha,\mathbf{x})$
        for all morphisms $(f,\phi,g)$ with $f\in\mathcal{R}_{n.d.}(f,\geqslant n+1)$,
        $g\in\mathcal{R}_{n.d.}(f,n)$.

        By definition of the attaching map $e$, the above diagram is a push-out diagram.
    \end{proof}

    Applying Lemma \ref{lemglue} inductively on the construction of $\mathbb{C}_{\mathscr{U}}X$
    described in Lemma \ref{lemonCX}, we get the following result.

    \begin{prop}\label{propofCX}
        Let $\nu: \mathscr{C}\to\mathscr{C'}$ be a morphism between 
        $E_\infty$ ring operads with corresponding morphism between monads
        $\nu:\mathbb{C}_{\mathscr{U}}\to \mathbb{C'}_{\mathscr{U}}$. Then  
        $$\nu:\mathbb{C}_{\mathscr{U}}X\to \mathbb{C'}_{\mathscr{U}}X$$
        is a homotopy equivalence for all space $X$.
    \end{prop}
    \begin{proof}
        By Definition \ref{defnEinfty}, both the following maps 
        \begin{align*}
            \nu&: \biggl(\displaystyle\coprod_{g\in\mathcal{R}_{n.d.}(f,n)}\mathscr{C}(g)\biggr)
            \to \biggl(\displaystyle\coprod_{g\in\mathcal{R}_{n.d.}(f,n)}\mathscr{C}'(g)\biggr)\\
            \nu&: L(f,n)_{\mathscr{C}}\to L(f,n)_{\mathscr{C}'}
        \end{align*}
        are equivalences. So applying 
        Lemma \ref{lemglue} and Lemma \ref{lemonCX},
        by induction on $n$, this proposition holds.
    \end{proof}

    Proposition \ref{propofCX} is also true for the monad $\mathbb{C}$ defined in Definition
    \ref{defnofmonad}.
    To see this, 
    note that if some non-degenerate $f$ is of type $(l;k_1,\cdots,k_l)$
    with $ (f,\sigma,\sigma_*f) $ a singular morphism
    and $\sigma_*f$
    is of type $(l';k'_1,\cdots,k'_{l'})$,
    then $k_1+\cdots+k_l> k'_1+\cdots+k'_{l'}$
    because we have evaluated some variable in $f$ to be zero.
    This gives a filtration on $\mathbb{C}X$ so that we can again 
    apply Lemma \ref{lemglue} inductively.

    We denote the full sub-category of $\widehat{\mathcal{R}}_{eff}$ generated by objects
    of the same type as some special $f$ to be $\widehat{\mathcal{R}}_{eff}(f)$ 
    with objects $\mathcal{R}(f)$. Moreover, we denote
    \begin{align*}
        \widehat{\mathcal{R}}^{=n (\leqslant n, <n, \text{resp.})}_{n.d.}&:=\coprod_{f \text{ special, }|f|=n (\leqslant n, <n, \text{resp.})} \widehat{\mathcal{R}}_{n.d.}(f)\\
        \widehat{\mathcal{R}}^{=n (\leqslant n, <n, \text{resp.})}_{eff}&:=\coprod_{f \text{ special, }|f|=n (\leqslant n, <n, \text{resp.})} \widehat{\mathcal{R}}_{eff}(f)\\
        \mathcal{R}^{=n (\leqslant n, <n, \text{resp.})}_{n.d.}&:=\coprod_{f \text{ special, }|f|=n (\leqslant n, <n, \text{resp.})} \mathcal{R}_{n.d.}(f)\\
        \mathcal{R}^{=n (\leqslant n, <n, \text{resp.})}_{eff}&:=\coprod_{f \text{ special, }|f|=n (\leqslant n, <n, \text{resp.})} \mathcal{R}_{eff}(f)\\
    \end{align*}

    \begin{prop}\label{propofCXreduced}
        Let $\nu: \mathscr{C}\to\mathscr{C'}$ be a morphism between 
        $E_\infty$ ring operads with corresponding morphism between monads
        $$\nu:\mathbb{C}\to \mathbb{C'}.$$ Then  
        $$\nu:\mathbb{C}X\to \mathbb{C'}X$$
        is a homotopy equivalence for all space $X$.
    \end{prop}
    \begin{proof}
        Let 
        $$F_n\mathbb{C}X:=\mathrm{Im}(\mathscr{C}(\bullet)\times_{\widehat{\mathcal{R}}_{eff}^{\leqslant n,op}}X^\bullet \to \mathbb{C}X).$$
        Then there is a push-out diagram
        \begin{center}
            \begin{codi}
                \obj {|(0)|\mathscr{C}(\bullet)\times_{\widehat{\mathcal{R}}^{=n,op}_{n.d.}}sX^\bullet
                &[12em]|(1')| F_{n-1}\mathbb{C}X\\
                |(1)| \mathscr{C}(\bullet)\times_{\widehat{\mathcal{R}}^{=n,op}}X^\bullet=
                \mathscr{C}(\bullet)\times_{\widehat{\mathcal{R}}_{n.d.}^{=n,op}}X^\bullet&
                |(2)| F_{n}\mathbb{C}X\\
                };
                \mor 0 "e":-> 1';
                \mor 0 -> 1;
                \mor 1 -> 2;
                \mor 1' -> 2;
            \end{codi}
        \end{center}
        Here $sX^m$ consists of $(x_1,\cdots,x_m)$ with some 
        $x_i$ lies in the image of $S^0\to X$,
        $e(\alpha,\sigma^*\mathbf{x})=(\sigma_*\alpha,\mathbf{x})$,
        and all $\sim$ are generated by 
        $(\sigma_*\alpha,\mathbf{x})\sim(\alpha,\sigma^*\mathbf{x})$.
        
        Beginning with $F_0\mathbb{C}X=F_0\mathbb{C}'X=*$,
        applying Lemma \ref{lemglue} inductively, 
        it suffices to show both
        \begin{align*}
            \nu:&\mathscr{C}(\bullet)\times_{\widehat{\mathcal{R}}^{=n,op}_{n.d.}}sX^\bullet\to \mathscr{C}'(\bullet)\times_{\widehat{\mathcal{R}}^{=n,op}_{n.d.}}sX^\bullet,\\
            \nu:&\mathscr{C}(\bullet)\times_{\widehat{\mathcal{R}}_{n.d.}^{=n,op}}X^\bullet\to\mathscr{C}'(\bullet)\times_{\widehat{\mathcal{R}}^{=n,op}_{n.d.}}X^\bullet.
        \end{align*}
        are equivalences.
        The latter has been proved in Lemma \ref{lemonCX} while 
        the former follows with the same proof.
    \end{proof}

    Using the same filtration and applying Lemma \ref{lemglue}
    inductively, we can also prove the follow result.

    \begin{lem}\label{lemofCXX'}
        Let $\mathscr{C}$ be a ring operad in $\mathscr{U}$
        with associated monad $\mathbb{C}$
        and let $X\to X'$ be an equivalence in $\mathscr{U}_e$.
        Then so is $\mathbb{C}X \to \mathbb{C}X'$.
    \end{lem}

    Now applying \cite[Proposition 9.8, Corollary 11.10]{may1972geometry}
    we get 
    $$ \eta: X \to B(\mathbb{C},\mathbb{C},X) $$
    is a homotopy equivalence for all $\mathbb{C}$-algebra $X$ in $\mathscr{U}_e$
    with homotopy inverse $\varepsilon$.

    \begin{prop}\label{propcompare}
        Let $\nu: \mathscr{C}\to\mathscr{C'}$ be a morphism between 
        $E_\infty$ ring operads. Then the pull-back of action functor
        $\nu^*:\mathscr{C}'[\mathscr{U}_e]\to \mathscr{C}[\mathscr{U}_e]$
        induces an equivalence between homotopy categories.

        In particular, any $\mathscr{C}$-algebra $X$ is equivalent
        to some $\mathscr{C}'$-algebra $Y$ since $\nu^*$
        is the identity map on underlying pointed spaces.
    \end{prop}
    \begin{proof}
        We define 
        \begin{align*}
            \nu_*:\mathscr{C}[\mathscr{U}_e]&\to \mathscr{C}'[\mathscr{U}_e]\\
            X &\mapsto B(\mathbb{C}',\mathbb{C},X).
        \end{align*}
        Then we get the following equivalences
        \begin{center}
            \begin{codi}
                \obj {|(0)| X &[2em] |(1)| B(\mathbb{C},\mathbb{C},X) &[7em] |(2)| \nu^*B(\mathbb{C}',\mathbb{C},X)= \nu^*\nu_*X\\
                };
                \mor 0 "\eta":-> 1 "B(\nu{,}\mathrm{id}{,}\mathrm{id})":-> 2;
            \end{codi}
        \end{center}
        and 
        \begin{center}
            \begin{codi}
                \obj {|(0)| \nu_*\nu^*Y=B(\mathbb{C}',\mathbb{C},\nu^*Y)&[7em] |(1)| B(\mathbb{C}',\mathbb{C}',Y) &[2em] |(2)| Y.\\
                };
                \mor 0 "B(\mathrm{id}{,}\nu{,}\mathrm{id})":-> 1 "\varepsilon":-> 2;
            \end{codi}
        \end{center}

        Here $B(\nu{,}\mathrm{id}{,}\mathrm{id})$ and $B(\mathrm{id}{,}\nu{,}\mathrm{id})$
        are equivalences by Lemma \ref{lemofCXX'} together with 
        \cite[Theorem A.4]{may1972geometry}.
    \end{proof}

    Therefore, we get the Comparison Theorem.
    \begin{thm}[Comparison Theorem \ref{thmcompare1}]\label{thmcompare}
        Let $\mathscr{C},\mathscr{C'}$ be any two 
        $E_\infty$ ring operads. Then the homotopy categories
        of
        $\mathscr{C}'[\mathscr{U}_e]$ and $\mathscr{C}[\mathscr{U}_e]$
        are equivalent.

        Moreover, any $\mathscr{C}$-algebra $X$ is equivalent
        to some $\mathscr{C}'$-algebra $Y$.
    \end{thm}
    \begin{proof}
        Note that the projections $\mathscr{C}\times\mathscr{C'}\to \mathscr{C}$
        and $\mathscr{C}\times\mathscr{C'}\to \mathscr{C}$ are morphisms between
        $E_\infty$ ring operads, so the Comparison Theorem follows
        from Proposition \ref{propcompare}.
    \end{proof}

    \section*{Acknowledgement}
    It is a pleasure to thank my mentor, Peter May, for his 
    guidance and encouragement. His insightful advice greatly 
    contributed to this paper, particularly in the comparison between two 
    constructions of bipermutative categories to 
    $E_\infty$ ring spectra and in the details in the proof 
    of the Comparison Theorem.

    I would also like to thank Weinan Lin for his helpful
    comments and suggestions. The original idea of defining
    ring operads emerged from our early discussions, and 
    the naming of ring operads is also credited to him.

    \bibliographystyle{plain}
    \bibliography{Reference}

\begin{thebibliography}{10}

\bibitem{OperadsforSymmetricMonoidalCategories}
A.D. {Elmendorf}.
\newblock {Operads for Symmetric Monoidal Categories}.
\newblock {\em arXiv e-prints}, January 2023.

\bibitem{Gepner_2015}
David Gepner, Moritz Groth, and Thomas Nikolaus.
\newblock Universality of multiplicative infinite loop space machines.
\newblock {\em Algebraic $\&$ Geometric Topology}, 15(6):3107--3153, December
  2015.

\bibitem{jardine2020persistenthomotopytheory}
J.~F. Jardine.
\newblock Persistent homotopy theory, 2020.

\bibitem{johnson2021bimonoidal}
Niles Johnson and Donald Yau.
\newblock Bimonoidal categories, {$E_n$}-monoidal categories, and algebraic $ k
  $-theory.
\newblock {\em arXiv preprint arXiv:2107.10526}, 2021.

\bibitem{laplaza2006coherence}
Miguel~L Laplaza.
\newblock Coherence for distributivity.
\newblock In {\em Coherence in categories}, pages 29--65. Springer, 2006.

\bibitem{Lellig1973}
Joachim Lillig.
\newblock A union theorem for cofibrations.
\newblock {\em Archiv der Mathematik}, 24:410--415, 1973.

\bibitem{may1972geometry}
J.P. May.
\newblock {\em The Geometry of Iterated Loop Spaces}.
\newblock Lecture notes in mathematics, 271. Springer-Verlag, 1972.

\bibitem{May_1974}
J.P. May.
\newblock {\em $E_{\infty}$ spaces, group completions, and permutative
  categories}, pages 61--94.
\newblock London Mathematical Society Lecture Note Series. Cambridge University
  Press, 1974.

\bibitem{may1977ring}
J.P. May.
\newblock {$E_\infty$} ring spaces and {$E_\infty$} ring spectra.
\newblock {\em Lecture Notes in Mathematics}, 1977.

\bibitem{MAY1980299}
J.P. May.
\newblock Pairings of categories and spectra.
\newblock {\em Journal of Pure and Applied Algebra}, 19:299--346, 1980.

\bibitem{may1982multiplicative}
J.P. May.
\newblock Multiplicative infinite loop space theory.
\newblock {\em Journal of pure and applied algebra}, 26(1):1--69, 1982.

\bibitem{may1999concise}
J.P. May.
\newblock {\em A concise course in algebraic topology}.
\newblock University of Chicago press, 1999.

\bibitem{TheconstructionofE_infringspacesfrombipermutativecategories}
J.P. {May}.
\newblock {The construction of {$E_\infty$} ring spaces from bipermutative
  categories}.
\newblock {\em Geometry $\&$ Topology Monographs}, 16:283--330, 2009.

\bibitem{may_moperad}
J.P. May.
\newblock Operads, moperads, and bioperads.
\newblock 2024.

\end{thebibliography}

\end{document}